\documentclass[12pt]{amsart}  
\usepackage{setspace}  
\usepackage{float}
\usepackage{xcolor}
\usepackage{amsmath,amssymb, amsthm,graphicx}
\usepackage{verbatim,enumerate}

\def\green{\null}

\usepackage[capitalise]{cleveref} 

\newtheorem{theorem}{Theorem}
\newtheorem{corollary}[theorem]{Corollary}
\newtheorem{lemma}[theorem]{Lemma}
\newtheorem{proposition}[theorem]{Proposition}

\theoremstyle{remark}
\newtheorem{remark}[theorem]{Remark}

\theoremstyle{definition}

\newtheorem{claim}{Claim}
\AtEndEnvironment{proof}{\setcounter{claim}{0}}
\newenvironment{claimproof}{\noindent\textsc{Proof of claim}}{\hfill$\qed$}
\crefname{claim}{Claim}{Claims}

\def\tinyskip{\vspace{2pt}}

\newtheorem{case}{Case}
\AtEndEnvironment{proof}{\setcounter{case}{0}}

\newcounter{subcases}[case]

\renewcommand*\thesubcases{\arabic{subcases}}
\newenvironment{subcases}
{%
 \setcounter{subcases}{0}%
    \def\subcase
      {%
        \tinyskip
        \par\noindent
        \refstepcounter{subcases}%
        \textit{Subcase \thecase.\thesubcases.}              
      }%
}
{%
}
\usepackage{mathtools}

\renewcommand{\v}[1]{\boldsymbol{{#1}}}
\newcommand{\cvec}[1]{\begin{pmatrix}#1\end{pmatrix}}

\DeclarePairedDelimiter\floor{\lfloor}{\rfloor}
\DeclarePairedDelimiter\set{\{}{\}}

\DeclarePairedDelimiter\brac{[}{]}

\newcommand{\card}[1]{\ensuremath{\# #1}} 

\newcommand{\WP}{W\hskip -1.6pt P}

\newcommand{\Pm}{P\hskip -1.6pt m}
\newcommand{\TA}{T\hskip -1.6pt A}
\DeclareMathOperator{\conv}{conv} 
\DeclareMathOperator{\wed}{W}
 
\DeclareMathOperator{\pyr}{pyr}

\newcommand{\R}{\mathbb{R}}

\renewcommand{\ge}{\geqslant} 
\renewcommand{\le}{\leqslant}

\usepackage{anysize}
\marginsize{3cm}{3cm}{3cm}{3cm}

\begin{document}

\date{\today}
\title[A lower bound theorem for $d$-polytopes with at most $2d$ vertices]{A refined lower bound theorem for $d$-polytopes with at most $2d$ vertices}

\author{Guillermo Pineda-Villavicencio}
\author{Jie Wang}
\author{David Yost}
\address{School of Information Technology, Deakin University, Locked Bag 20000, Geelong VIC. 3220, Australia}
\email{\texttt{guillermo.pineda@deakin.edu.au}}

\address{Federation University, Mt. Helen, Vic. 3350, Australia}
\email{\texttt{wangjiebulingbuling@foxmail.com},\texttt{d.yost@federation.edu.au}}


%

\begin{abstract} In 1967, Gr\"unbaum  conjectured  that the function 
$$
\phi_k(d+s,d):=\binom{d+1}{k+1}+\binom{d}{k+1}-\binom{d+1-s}{k+1},\; \text{for  $2\le s\le d$}
$$
provides the minimum number of $k$-faces for a $d$-dimensional polytope (abbreviated as a $d$-polytope) with $d+s$ vertices. In 2021, Xue proved this conjecture for each $k\in[1\ldots d-2]$ and characterised the unique minimisers, each having $d+2$ facets. 

\green{In this paper, we refine Xue's theorem  by considering $d$-polytopes with  $d+s$ vertices ($2\le s\le d$) and at least $d+3$ facets. If $s=2$, then there is  precisely one minimiser for many values of $k$.  For other values of $s$, the number of $k$-faces is at least  $\phi_k(d+s,d)+\binom{d-1}{k}-\binom{d+1-s}{k}$, which is met  by precisely two polytopes in many cases, and up to five polytopes for certain values of $s$ and $k$}. We also characterise the minimising polytopes.
\end{abstract}

\maketitle


\section{Introduction}
\label{sec:intro}

 We denote by $f_{k}$ the number of $k$-faces in a polytope; the \textit{$f$-vector} of a $d$-polytope is the sequence $(f_0,\ldots, f_{d-1})$ of the number of faces in the polytope. Xue's theorem settling Gr\"unbaum's 1967 conjecture ~\cite[Sec.~10.2]{Gru03} reads as follows.

\begin{theorem}[$d$-polytopes with at most $2d$ vertices,  Xue 2021]\label{thm:at-most-2d}  
Let $d\ge 2$ and $1\le s \le d$. If $P$ is a $d$-polytope with $d+s$ vertices, then
\[f_k(P)\ge \phi_k(d+s,d),\;  \text{for all $k\in[1\ldots d-1]$}.\] 
Furthermore, for some $k\in[1\ldots d-2]$, there is a unique polytope whose number of $k$-faces equals $\phi_k(d+s,d)$,  and this minimiser has $d+2$ facets.
\end{theorem}

The structure of  $d$-polytope with $d+2$ facets is well understood; see, for instance, \cite[Sec.~3]{McMShe70} and \cref{lem:dplus2facets}. In lower bound theorems for  $d$-polytopes with $2d+1$ vertices \cite{PinYos22,Xue22} (see also \cref{thm:2dplus1}) and $2d+2$ vertices \cite{PinTriYos24}, there is a dichotomy among the minimisers: they have either $d+2$ or $d+3$ facets. Naturally, $(d-1)$-polytopes with at most $2(d-1)$ vertices prominently featured as potential facets in the proofs of these two theorems. Pineda-Villavicencio~\cite[Prob.~8.7.11]{Pin24} conjectured that this dichotomy extends to minimisers of the number of $k$-faces among $d$-polytopes with at most $3d-1$ vertices. In light of these reflections, we find it natural to extend this dichotomy to $d$-polytopes with at most $2d$ vertices by refining \Cref{thm:at-most-2d}. 

Since the case of $d$-polytopes with $d+2$ facets is completely covered by \cref{thm:at-most-2d}, our main theorem deals with the case of $d$-polytopes with $d+3$ or more facets. It  will be stated in more detail and proved in subsequent sections.


\green{\begin{theorem}[Refined theorem for $d$-polytopes with at most $2d$ vertices (Short version)]
\label{thm:at-most-2d-refined-short}  
Given parameters $d\ge 3$ and $2\le s \le d$, and a  $d$-polytope $P$ with $d+s$ vertices, the following statements hold: 
\begin{enumerate} [{\rm (i)}]
    \item If $s=2$ and $P$ has at least $d+\ell$ facets where $\ell\in [3\ldots d]$, then $f_k(P)\ge \phi_k(d+2,d)+\binom{d-1}{k-1}-\binom{d-\ell+1}{k-\ell+1}$ for all $k\in[1\ldots d-2]$. Moreover, for each such $\ell$ and $k\ge \ell-1$ there is a unique polytope  whose number of $k$-faces equals this lower bound.
    \item If $s\in [3\ldots d]$ and $P$ has at least $d+3$ facets, then $f_k(P)\ge\zeta_k(d+s,d)$.  There are two fundamental examples for which the number of $k$-faces equals $\zeta_k(d+s,d)$ for each $k\in[1\ldots d-2]$. Additionally, if  $s<d$ and $k\ge d-s+2$, then a third minimiser arises, and when $s=4$, two further examples appear.  
\end{enumerate}
\end{theorem}}

A full version of \cref{thm:at-most-2d-refined-short}, describing the minimising polytopes in detail, is provided at the beginning of \cref{sec:proofs}.

Since this expression will appear frequently, it is convenient to define 
\begin{equation*}
\zeta_k(d+s,d):=\phi_k(d+s,d)+\binom{d-1}{k}-\binom{d+1-s}{k}.
\end{equation*}


\section{Minimisers}
\label{sec:minimisers}
This section presents the minimisers in \cref{thm:at-most-2d,thm:at-most-2d-refined-short}.

When $k\in [1\ldots d-2]$, each minimiser in \cref{thm:at-most-2d} 
is a $(d-s)$-fold pyramid over a simplicial $s$-prism for $s\in [1\ldots d]$; this polytope is called the \textit{$(s,d-s)$-triplex} and denoted by $M(s,d-s)$. Triplices were introduced by Pineda-Villavicencio et.~al \cite[Sec.~3]{PinUgoYos15}. Each triplex has $d+2$ facets.  McMullen and Shephard \cite[Sec.~3]{McMShe70} provided expressions for  the  number of faces of such polytopes.

\begin{lemma}[McMullen and Shephard 1970]\label{lem:dplus2facets}
Let  $P$ be a $d$-dimensional polytope with $d+2$ facets, where $d\ge 2$. Then, there exist  integers $2\le a\le d$ and $1\le m\le \floor{a/2}$ such that $P$ is a $(d-a)$-fold pyramid over $T(m)\times T(a-m)$. The number of  $k$-faces of $P$ is

\begin{equation}\label{eq:dplus2facets}
\binom{d+2}{k+2} -\binom{d-a+m+1}{k+2}-\binom{d+1-m}{k+2} +\binom{d-a+1}{k+2}.
\end{equation}
In particular,  $f_{0}(P)=d+1+m(a-m)$.
\end{lemma}
Recall that the {\it Cartesian product} of a $d$-polytope $P\subset \R^{d}$ and a $d'$-polytope $P'\subset \R^{d'}$ is the Cartesian product of the sets $P$ and $P'$: 
\begin{equation*}
P\times P'=\set*{(p, p')^{t}\in \R^{d+d'}\mid p\in P,\, p'\in P}.
\end{equation*}

Dually to \cref{lem:dplus2facets}, we have a lemma for $d$-polytopes with $d+2$ vertices, which is from Gr\"unbaum~\cite[Sec.~6.1]{Gru03}.  The {\it direct sum} $P\oplus P'$ of a $d$-polytope $P\subset \R^{d}$ and a $d'$-polytope $P'\subset\R^{d'}$ with the origin in their relative interiors is the $(d+d')$-polytope: 
\begin{equation*}
\label{eq:direct-sums}
P\oplus P'=\conv \left(\left\{\cvec{\v p\\\v 0_{d'}}\in \R^{d+d'}\middle|\; \v p\in P\right\}\bigcup \left\{\cvec{\v 0_{d}\\\v p'}\in \R^{d+d'}\middle|\; \v p'\in P'\right\}\right).    
\end{equation*} 

 \begin{lemma}\label{lem:dplus2vertices}
 Let  $P$ be a $d$-dimensional polytope with $d+2$ vertices, where $d\ge 2$. Then, there exist  integers $2\le a\le d$ and $1\le m\le \floor{a/2}$ such that $P$ is a $(d-a)$-fold pyramid over $T(m)\oplus T(a-m)$. The number of  $k$-dimensional faces of $P$ is
 \begin{equation*}
 \binom{d+2}{d-k+1} -\binom{d-a+m+1}{d-k+1}-\binom{d-m+1}{d-k+1} +\binom{d-a+1}{d-k+1}.
 \end{equation*}
 In particular,  $f_{d-1}(P)=d+1+m(a-m)$.
 \end{lemma}

As Gr{{\"u}}nbaum~\cite[Thm.~6.4]{Gru03}, we denote the $(d-a)$-fold pyramid over $T(m)\oplus T(a-m))$ by $T_{m}^{d,d-a}$.  
    


 Gr{{\"u}}nbaum~\cite[p.~101]{Gru03} established inequalities for the number of $k$-faces in $d$-polytopes with $d+2$ vertices.

\begin{lemma} For $k\in [0\ldots d-1]$, the following hold:
\begin{enumerate}[{\rm (i)}]
\item If $2\le a\le d$ and $1\le m\le \floor{a/2}-1$, then  $f_{k}(T_{m}^{d,d-a})\le f_{k}(T_{m+1}^{d,d-a})$, with strict inequality if and only if $m\le k$.
\item If $2\le a\le d-1$ and $1\le m\le \floor{a/2}$, then  $f_{k}(T_{m}^{d,d-a})\le f_{k}(T_{m}^{d,d-(a+1)})$, with strict inequality if and only if $a-m\le k$.
\end{enumerate}
\label{lem:dplus2-vertices-inequalities}
\end{lemma}

 

%

\green{A vertex $v$ in a $d$-polytope $P$ is \textit{simple} if it is contained in exactly  $d$ edges. Otherwise, it is \textit{nonsimple}. A nonsimple vertex in $P$ may be simple in a proper face of  $P$, and we often need to make this distinction. If $K$ is a closed halfspace in $\mathbb{R}^d$ such that $v$ is the only vertex of $P$ not in $K$, we say that a polytope $P'$ is obtained by \textit{truncating $v$} if $P'=P\cap K$.}

If we truncate a simple vertex from the $(2,d-2)$-triplex ($d\ge2$), we obtain a $d$-polytope $\Pm(d)$ whose number of faces is 
\begin{equation}
f_k(\Pm(d))=\begin{cases}
2d+1&\text{if $k=0$};\\
\binom{d+1}{k+1}+\binom{d}{k+1}+\binom{d-1}{k}& \text{if $k\in [1\ldots d-1]$}.
\end{cases}
\label{eq:pentasm-function}
\end{equation}
Pineda-Villavicencio and Yost~\cite{PinYos22}, also \cite{PinUgoYos16a}, called the $d$-polytope $\Pm(d)$ the \textit{$d$-pentasm}.

\begin{theorem}[{\cite[Thm.~25]{PinYos22}}, {\cite[Thm.~5.1]{Xue22}}]
\label{thm:2dplus1}
Let $d\ge 4$, $P$  a $d$-polytope with \textbf{at least} $2d+1$ vertices, and $k\in [1\ldots d-2]$. 
\begin{enumerate}[{\rm (i)}]
\item  If $P$ has at least $d+3$ facets, then $f_{k}(P)\ge f_{k}(\Pm(d))$, with equality for some $k\in [1\ldots d-2]$ only if  $P=\Pm(d)$.
\item  If $P$ has  $d+2$ facets, then $f_{k}(P)\ge f_{k}((T_{2}^{d,d-(\floor{d/2}+2)})^{*})$. If $d$ is even  then $(T_{2}^{d,d-(\floor{d/2}+2)})^{*}$ has $2d+1$ vertices and this bound can be attained.  
\end{enumerate}
\end{theorem}

Since it appears frequently, it will be convenient to denote the $a$-fold pyramid over a $b$-pentasm, by $\Pm(b,a)$. We will denote as usual the pyramid over an arbitrary polytope $P$ by $\pyr(P)$ and the $t$-fold pyramid over $P$ by $\pyr_t(P)$. 

The number of $k$-faces in the polytope $\Pm(s-1, d+1-s)$ is obtained by counting the $k$-faces among the $d+1-s$ apices of the pyramid and the $k$-faces obtained from an $i$-face in an $(s-1)$-pentasm and $k-i$ vertices from the $d+1-s$ apices of the pyramid, \green{$0 \le i \le k$}. That is,  for $k\in [0\ldots d-1]$ we have that
\begin{align*}
f_k(\Pm(s-1, d+1-s))&=\binom{d+1-s}{k+1}+(2s-1)\binom{d+1-s}{k} \\
 &\quad+\sum_{i=1}^{k}\binom{d+1-s}{k-i} \left[\binom{s}{i+1}+\binom{s-1}{i+1}+\binom{s-2}{i} \right].
\end{align*}

We write this expression in terms of the function $\phi_{k}(d+s,d)$ defined in the abstract. 
\begin{align*}
\binom{d+1}{k+1}&=\binom{d+1-s}{k+1}+s\binom{d+1-s}{k}+\sum_{i=1}^{k}\binom{d+1-s}{k-i} \binom{s}{i+1}\\
\binom{d}{k+1}-\binom{d+1-s}{k+1}&=(s-1)\binom{d+1-s}{k}+\sum_{i=1}^{k}\binom{d+1-s}{k-i} \binom{s-1}{i+1}\\
\binom{d-1}{k}-\binom{d+1-s}{k}&=\sum_{i=1}^{k}\binom{d+1-s}{k-i} \binom{s-2}{i}
\end{align*}
Comparing these expressions with the definition of $\zeta_k$, we arrive at the following relation.   
\begin{equation}\label{eq:pyramid-pentasm-function}
f_k(\Pm(s-1, d+1-s)) =\zeta_k(d+s,d).
\end{equation}

  The wedge construction will be very useful to us. Let $P$ be a $d$-polytope embedded in the hyperplane $x_{d+1}=0$ of $\R^{d+1}$. Let $F$ be a proper face of $P$, and let $C$ be the halfcylinder $P\times [0,\infty)\subset \R^{d+1}$. We cut the halfcylinder with a hyperplane $H'$ through $F\times \set{0}$ so that  $C$ is partitioned into two parts, one bounded and  one unbounded. The \textit{wedge} of $P$ at $F$ is  the bounded part; it is denoted by $\wed_{F}(P)$. The sets $P$ and $H'\cap C$, the \textit{bases} of $\wed_{F} (P)$, define facets of $\wed_{F} (P)$ that are combinatorially isomorphic to $P$ and intersect at the face $F\times \set{0}$. The wedge $W$  over a $d$-polytope $P\times \set{0}\subseteq \R^{d+1}$ at a face $F\times \set{0}$ of $P\times \set{0}$ is combinatorially isomorphic to a prism $Q$ over $P\times \set{0}$ where the face prism $(F\times \set{0})$ of $Q$ has collapsed into $F\times \set{0}$. \green{For further information, refer to \cite[Sec.~2.6]{Pin24}}. The following basic facts about wedges will be used several times. Their proofs are routine.

\begin{lemma}\label{wedge}
    Let $\wed_{F}(P)$ be a wedge of a $d$-polytope $P$ at a face $F$ of $P$.
    \begin{enumerate}[(i)]
        \item A $k$-face of $\wed_{F}(P)$ is either a $k$-face of one of the bases of $\wed_{F}(P)$,
        or the wedge of a $(k-1)$-face $F'$ of P at the proper face $F\cap F'$, or a prism over a $(k-1)$-face of $P$ disjoint from $F$.
        \item If $F$ is a facet, then for each value of $k\in [0,\ldots,d+1]$ we have
        $$f_k(\wed_{F}(P))=2f_k(P)+f_{k-1}(P)-f_k(F)-f_{k-1}(F).$$
        \item If $W_1,P_1,F_1$ denote pyramids over $\wed_{F}(P),P,F$, respectively, then $W_1$ is combinatorially equivalent to $\wed_{F_1}(P_1)$.
    \end{enumerate}
\end{lemma}

\begin{figure}             
\begin{center}     
\includegraphics[scale=1.2]{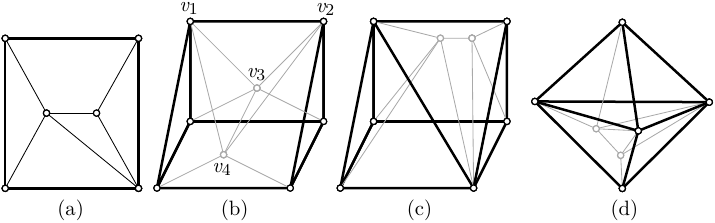}  
\end{center}
\caption{Schlegel digrams of polytopes. {(a)}  The tetragonal antiwedge. {(b)} The 4-polytope $\WP$. (c) The 4-polytope $\TA(4)$. (d) The 4-polytope $Z(4)$.}
\label{fig:polytopes}   
\end{figure} 
We let $\TA(2)$ be a quadrilateral and let $\TA(3)$ be the \textit{tetragonal antiwedge},  the unique nonpyramidal 3-polytope with six vertices and six 2-faces (four triangles and two quadrilaterals); see \cref{fig:polytopes}(a). For $d\ge 4$, we iteratively define  the $d$-polytope $\TA(d)$ as the wedge of $\TA(d-1)$ at a facet $\TA(d-2)$ of $\TA(d-1)$. The polytope $\TA(d)$ has $\TA(d-1)$ as a facet, and thus this construction is sound. See \cref{fig:polytopes}(c). 
\begin{remark}[Structure of $\TA(d)$]
\label{rmk:structure-TA}
\green{The polytope $\TA(d)$ has $2d$ vertices. It has one vertex $v$ with maximum degree $2d-2$, $d-2$ vertices with degree $d+1$, and every other vertex is simple}. Its $d+3$ facets are the following:  
\begin{enumerate}[{\rm (i)}] 
    \item $d-2$ copies of $\TA(d-1)$,
    \item three simplices,
    \item one simplicial $(d-1)$-prism,
    \item one $M(d-2,1)$.
\end{enumerate}
The facets not containing $v$ are the prism and one of the simplices.
\end{remark}

The polytope $\TA(d)$  and a pyramid over a $(d-1)$-pentasm have the same $f$-vector. This follows from noticing that the pyramid over the pentagon (the 2-pentasm) and the tetragonal antiwedge $\TA(3)$ both have $f$-vector $(6,10,6)$ and that a pyramid over a $(d-1)$-pentasm is a wedge of a pyramid over a $(d-2)$-pentasm at a facet that is itself a pyramid over $(d-3)$-pentasm. As a consequence, a $(d-s)$-fold pyramid over the $s$-polytope $\TA(s)$, denoted by $\TA(s,d-s)$, and the $d$-polytope $\Pm(s-1, d+1-s)$ both have the same $f$-vector: 
\begin{equation}\label{eq:pyramid-iterative-wedge-function}
f_k(\TA(s, d-s)) =\zeta_k(d+s,d).
\end{equation}

We will see that, for fixed $s\in[3\ldots d]$, the polytope $\TA(s,d-s)$ and $\Pm(s-1, d+1-s)$ minimise the number of $k$-faces for every $k\in [1\ldots d-2]$. For many values of $s$ and $k$, these are the only two minimisers. But for some values of $s$ and $k$, there may be up to three more.

Consider $s=d=4$. The polytopes $\TA(4)$ and $\Pm(3, 1)$ both have $f$-vector $(8,18,17,7)$. The wedge of a square pyramid at one its edges also has $f$-vector $(8,18,17,7)$; the combinatorial type depends on whether or not the edge contains the apex of the pyramid. If the edge is in the base of the pyramid, we obtain the example called Polytope 2 in \cite[Lemma 18(ii)]{PinUgoYos22};  \green{here we denote this 4-polytope by $\WP$.} If the edge contains the apex, then \cref{wedge} tells us the polytope is a pyramid over the wedge of a square at one of the vertices; we use $\Sigma(3)$  to denote the wedge of a square at one of the vertices\footnote{There are two natural generalisations of $\Sigma(3)$ to higher dimensions, denoted  $\Sigma(d)$ and $C(d)$ in \cite{PinUgoYos15}. They both have $3d-2$ vertices and are not relevant in this paper.}.  Therefore for $s=4$, $\pyr_{d-4}(\WP)$ (the $(d-4)$-fold pyramid over $\WP$), and  $\pyr_{d-3}(\Sigma(3))$ (the $(d-3)$-fold pyramid over $\Sigma(3)$), also minimise the number of $k$-faces for every $k\in [1\ldots d-2]$.

Let $Z(3)$ be a triangular bipyramid and let $Z(2)$ be a triangular face of $Z(3)$. For $d\ge 4$, we define $Z(d)$ inductively by $Z(d+1):=\wed_{Z(d-1)}(Z(d))$.   
\begin{remark}[Structure of $Z(d)$]
\label{rmk:structure-Z} The polytope $Z(d)$ is $d$-dimensional with $2d-1$ vertices and $d+3$ facets. It has two vertices, say $v$ and $w$, with maximum degree $2d-2$,  $d-2$ vertices with degree $d+1$, and $d-1$ simple vertices. The $d+3$ facets are
\begin{enumerate}[{\rm (i)}] 
    \item $d-2$ copies of $Z(d-1)$,
    \item three simplices,
    \item two copies of $M(d-2,1)$.
\end{enumerate}
The facets not containing $v$ are one of the $M(d-2,1)$ copies and one of the simplices. The facets not containing $w$ are the other $M(d-2,1)$ copy and another one of the simplices. See \cref{fig:polytopes}(d).
\end{remark}

The next result says that the polytope $Z(d)$ minimises $f_k$ among $d$-polytopes with $2d-1$ vertices and at least $d+3$ facets only for $k\ge3$.

\begin{lemma} For $d\ge 2$, the $f$-vector of the polytope $Z(d)$ is given by 
\label{lem:fvector-Z}
\begin{equation*}
f_{k}=\begin{cases}
2d-1,&\text{if $k=0$};\\
d^2+d-3=\zeta_1(2d-1,d)+1,&\text{if $k=1$};\\
\binom{d+1}{k+1}+\binom{d}{k+1}+\binom{d-1}{k},&\text{if $k\ge 2$}.
\end{cases}
\end{equation*}
In particular, $f_2(Z(d))=\zeta_2(2d-1,d)+1$, and $f_k(Z(d))=\zeta_k(2d-1,d)$ for $k\ge3$.
\end{lemma}

\begin{proof}
Since $Z(3)$ has exactly two more vertices than $Z(2)$, the expression for $f_{0}$ is clear. An induction argument on $d\ge 3$ establishes that $f_1(Z(d))=(d-2)^2+5(d-2)+3$. Similarly, the expression  $f_2(Z(d))=\frac{1}{3}(d^3-4d+3)$ can be proved by induction on $d$, starting  with the base cases $d=2,3$; the induction step uses the expression for $f_{1}$ and \cref{wedge}. 

For computing $f_{k}$ for $k\ge3$, we  proceed again by induction on $d\ge 3$ for all $k$, with $d=4,3$ as the base cases. Suppose the conclusion is valid up to dimension $d$. Repeated applications of Pascal's identity yield
    \begin{align*}
f_k(Z(d+1))&=2f_k(Z(d))-f_k(Z(d-1))+f_{k-1}(Z(d))-f_{k-1}(Z(d-1))\\
&=2\binom{d+1}{k+1}+2\binom{d}{k+1}+2\binom{d-1}{k}-\binom{d}{k+1}-\binom{d-1}{k+1}-\binom{d-2}{k}\\
&\qquad+\binom{d+1}{k}+\binom{d}{k}+\binom{d-1}{k-1}-\binom{d}{k}-\binom{d-1}{k}-\binom{d-2}{k-1}\\
&=2\binom{d+1}{k+1}+\binom{d+1}{k}+\binom{d}{k+1}-\binom{d-1}{k+1}+\\
&\qquad+\binom{d-1}{k}-\binom{d-2}{k}+\binom{d-1}{k-1}-\binom{d-2}{k-1}\\
&=2\binom{d+1}{k+1}+\binom{d+1}{k}+\binom{d-1}{k}+\binom{d-2}{k-1}+\binom{d-2}{k-2}\\
&=\binom{d+1}{k+1}+\binom{d+2}{k+1}+\binom{d-1}{k}+\binom{d-1}{k-1}\\
 &=\binom{d+2}{k+1}+\binom{d+1}{k+1}+\binom{d}{k}.
 \end{align*}
Hence, the conclusion is also valid in dimension $d+1$.
   \end{proof}

We are also interested in the $(d-s-1)$-fold pyramid over $Z(s+1)$, denoted by $Z(s+1,d-s-1)$, which has $d+s$ vertices and $d+3$ facets. 
If $P$ is a pyramid over a facet $F$, then $f_k(P)=f_k(F)+f_{k-1}(F)$ for each $k\ge 0$. Iteration of this leads to the  following conclusion for the $(d-s-1)$-fold pyramid over $Z(s+1)$: 

\begin{enumerate}[{\rm (i)}] 
\item If $s<d-1$, then the expression for $f_k$ only involves  terms of the form $f_j$ with $3\le j\le k$, leading to $f_k(Z(s+1,d-s-1))=\zeta_k(d+s,d)$ for $k\ge d-s+2$. 

\item However, for $k\le d-s+1$, we have $f_k(Z(s+1,d-s-1))=\zeta_k(d+s,d)+1$. \end{enumerate}
   Thus, $Z(s+1,d-s-1)$ minimises the number of $k$-faces among $d$-polytopes with $d+s$ ($3\le s \le d$) vertices only for $k\in[d-s+2,d-1]$.

\section{Proof of the main theorem}
\label{sec:proofs}

We list the full version of the main theorem.

\green{\begin{theorem}[Refined theorem for $d$-polytopes with at most $2d$ vertices]
\label{thm:at-most-2d-refined}  
Given parameters $d\ge 3$ and $2\le s \le d$, and a  $d$-polytope $P$ with $d+s$ vertices, the following statements hold: 
\begin{enumerate} [{\rm (i)}]
    \item If $s=2$, then  $P$ has the form $T_{m}^{d,d-a}$ for some $2\le a\le d$ and $1\le m\le \floor{a/2}$. If $P$ has at least $d+\ell$ facets where $\ell\in [3\ldots d]$, then, for all $k\in[1\ldots d-2]$,  
\begin{equation*}
f_k(P)\ge \phi_k(d+2,d)+\binom{d-1}{k-1}-\binom{d-\ell+1}{k-\ell+1}.
\end{equation*} 
Moreover, if $m=1$ and $P\ne T_{1}^{d,d-\ell}$, then $f_k(P) >f_k(T_{1}^{d,d-\ell})$, for each such $\ell$ and $k\ge \ell-1$. If $m\ge 2$, then $f_k(P) >f_k(T_{1}^{d,d-\ell})$, for each $k\in [1\ldots d-2]$.
    \item If $s\in [3\ldots d]$ and $P$ has at least $d+3$ facets, then $f_k(P)\ge\zeta_k(d+s,d)$. Additionally, if $k\in[1\ldots d-2]$ is fixed, and $P$ is a $d$-polytope whose number of $k$-faces equals $\zeta_k(d+s,d)$, then either
\begin{enumerate}
\item $P$ is 
a $(d-s+1)$-fold pyramid over a $(s-1)$-pentasm, or
\item $P$ is a  $(d-s)$-fold pyramid over $\TA(s)$, or 
\item $s<d$, $k\ge d-s+2$, and $P$ is a $(d-s-1)$-fold pyramid over $Z(s+1)$, or 
\item $s=4$, and $P$ is a $(d-3)$-fold pyramid over $\Sigma(3)$, 
\item $s=4$, and $P$ is a $(d-4)$-fold pyramid over $\WP$. 
\end{enumerate}
\end{enumerate}
\end{theorem}}


We start from the beginning, at $d$-polytopes  with $d+2$ vertices. This theorem is a variation of the dual statement in \cref{thm:at-most-2d}. 

\begin{theorem}[{$d+2$ vertices}]
\label{thm:dplus2-vertices}	Let $d\ge 3$, $2\le \ell,a\le d$, and $1\le m\le \floor{a/2}$. Let $P:=T_{m}^{d,d-a}$ be a $d$-polytope  with  \textbf{at least} $d+\ell$ facets,  other than $T_{1}^{d,d-\ell}$.  The following hold:
\begin{enumerate}[{\rm (i)}]
    \item If $m=1$, then $f_k(P) >f_k(T_{1}^{d,d-\ell})$, for each $k\ge \ell-1$.
    \item If $m\ge 2$, then $f_k(P) >f_k(T_{1}^{d,d-\ell})$, for each $k\in [1\ldots d-2]$.
\end{enumerate}
\end{theorem}
\begin{proof} 

 If $m=1$, \cref{lem:dplus2vertices} gives $f_{d-1}(P)=d+1+m(a-m)=d+a$. Since $P\ne T_{1}^{d,d-\ell}$, we have that $a>\ell$, in which case \cref{lem:dplus2-vertices-inequalities}(ii) gives that $f_k(P) >f_k(T_{1}^{d,d-\ell})$, for each $k\ge \ell-1$. Henceforth, we assume that $m\ge 2$. 
We prove (ii) by an induction argument on $d\ge 3$ for all $a,\ell\in[2\ldots d]$ and $2\le m\le \floor{a/2}$. The case $d=3$ is simple. We then assume that $d\ge 4$,  $f_{d-1}(P)\ge d+\ell$, and that (ii) holds for all $(d-1)$-polytopes $T_{m'}^{d-1,d-a'-1}$ with  at least $d-1+\ell'$ facets for all $a',\ell'\in[2\ldots d-1]$ and $2\le m'\le \floor{a'/2}$. 

\cref{lem:dplus2-vertices-inequalities}(i) yields that $f_{k}(T_1^{d,d-\ell})<f_{k}(T_2^{d,d-\ell})$, for each $k\ge 1$.
Additionally, when $\ell\le a$, \cref{lem:dplus2-vertices-inequalities} gives that
\begin{equation*}
\label{eq:2dplus2-dplus2-facets-2}
f_{k}(T_2^{d,d-\ell})\le f_{k}(T_2^{d,d-a})\le f_{k}(T_m^{d,d-a}).
\end{equation*}
Hence, for $k\in [1\ldots d-2]$ and $\ell\le a$, we get that $f_{k}(T_1^{d,d-\ell})<f_{k}(T_m^{d,d-a})$. 

We now assume that $d\ge \ell>a$. Hence $P$ is a pyramid with base $F:=T_m^{d-1,d-1-a}$, where $a\in[2\ldots d-1]$ and $m\ge 2$. The induction hypothesis holds for $F$ when $\ell\in [2\ldots d-1]$, since $F$ has $d-1+2$ vertices and at least $d-1+\ell$ $(d-2)$-faces. As a consequence, for each $k\in [1\ldots d-2]$, we find that
 \begin{align*}
 f_{k}(P)&=f_{k}(F)+f_{k-1}(F)\\
 &> f_{k}(T_{1}^{d-1,d-1-\ell})+f_{k-1}(T_{1}^{d-1,d-1-\ell})\\
 &=\brac*{\binom{d+1}{d-k}-\binom{d-\ell+1}{d-k}-\binom{d-1}{d-k}+\binom{d-\ell}{d-k}}\\
 &\quad + \brac*{\binom{d+1}{d-k+1}-\binom{d-\ell+1}{d-k+1}-\binom{d-1}{d-k+1}+\binom{d-\ell}{d-k+1}}\\
 &=f_k(T_{1}^{d,d-\ell}).
 \end{align*}
For the case $f_{d-2}(P)$, we remark that  $f_{d-2}(F)\ge d-1+\ell=f_{d-2}(T_{1}^{d-1,d-1-\ell})$ by assumption and that $f_{d-3}(F)>f_{d-3}(T_{1}^{d-1,d-1-\ell})$ by induction.

Suppose $\ell=d$. Then $f_{d-2}(F)\ge d-1+\ell\ge 2(d-1)$, and thus, for each $k\in [1\ldots d-2]$ the induction hypothesis on $d-1$ for $\ell'= d-1$ gives that 
	 \begin{align*}
 f_{k}(P)&=f_{k}(F)+f_{k-1}(F)\\
 &\ge f_{k}(T_{1}^{d-1,0})+1+f_{k-1}(T_{1}^{d-1,0})+1\\
 &=\brac*{\binom{d+1}{d-k}-\binom{2}{d-k}-\binom{d-1}{d-k}+\binom{1}{d-k}+1}\\
 &\quad + \brac*{\binom{d+1}{d-k+1}-\binom{2}{d-k+1}-\binom{d-1}{d-k+1}+\binom{1}{d-k+1}+1}\\
 &=\binom{d+2}{d-k+1}-\binom{3}{d-k-1}-\binom{d}{d-k+1}+\binom{2}{d-k+1}+2\\
 &>\binom{d+2}{d-k+1}-\binom{1}{d-k}-\binom{d}{d-k+1}=f_k(T_{1}^{d,0}).
 \end{align*}
This completes the proof of the lemma.
\end{proof}





\begin{proposition}[{\cite[Prop.~3.1]{Xue21}}]\label{prop:number-faces-outside-facet}
Let $d\ge 2$ and let $P$ be a $d$-polytope. In addition, suppose that  $r\le d+1$ is given and  that $S:=(v_1,v_2,\ldots,v_r)$ is a sequence of distinct vertices in $P$. Then the following hold.
\begin{enumerate}[{\rm (i)}]
    \item  There is a sequence $F_1, F_2,\ldots, F_r$ of faces of $P$ such that each $F_i$ has dimension $d-i+1$ and contains $v_i$, but does not contain any $v_j$ with $j<i$.
    \item For each $k\ge1$, the number of $k$-faces of $P$ that contain at least one of the vertices in $S$ is bounded from below by 
\begin{equation*}
\sum_{i=1}^{r}f_{k-1}(F_i/v_i)\ge  \sum_{i=1}^{r} \binom{d+1-i}{k}
\end{equation*}
\item In case of equality in (ii) for some $k\in [1\ldots d-2]$, then for any ordering $v_{\ell1},\ldots,v_{\ell r}$ of the vertices of $S$, we must have that the number of $k$-faces containing $v_{\ell i}$ and not containing any vertex $v_{\ell j}$ with $j<i$ is precisely $\binom{d-i+1}{k}$. In particular, every vertex in $S$ is simple.
\end{enumerate}
\end{proposition}
\begin{proof} Statements (i)--(ii) are from \cite[Prop.~3.1]{Xue21}. Part  (iii) is proven in \cite[Cor.~8.2.6]{Pin24}; it follows from noting that the minimum of the right-hand side is attained when each vertex figure $F_{i}/v_{\ell i}$ is a $(d-i)$-simplex. 
\end{proof}



We now deal with dimensions 3 and 4 in \cref{thm:at-most-2d-refined}.

\begin{proposition}\label{basecase}
    For  $d=3$ or $4$ and $3\le s\le d$, the $d$-polytopes with precisely  $d+s$ vertices, at least $d+3$ facets, and at most $f_k(\Pm(s-1,d+1-s))$  $k$-faces for each $k\in [1\ldots d-2]$ are as follows.

\begin{enumerate}[{\rm (i)}]
\item For $d=3$, the minimisers are $\TA(3)$ and the pyramid over the pentagon, both with $f$-vector $(6, 10, 6)$.
\item For $d=4$, the minimisers with 7 vertices are  the pyramid over $\TA(3)$ and the 2-fold pyramid over a pentagon, both with $f$-vector $(7,16,16,7)$. 
\item The minimisers with 8 vertices, all with $f$-vector $(8,18,17,7)$, are the polytopes $\TA(4)$, $\Pm(3,1)$, $\WP$, and $\pyr(\Sigma(3))$.
\end{enumerate}
\label{prop:small-cases}
\end{proposition}    
\begin{proof}
  The catalogues of 3-polytopes with 6 vertices (see, for instance,  \cite[Fig.~3]{BriDun73}) reveal that the two possible minimisers are  $\TA(3)$ and the pyramid over the pentagon; both have $f$-vector $(6, 10, 6)$. This can  be deduced easily from Steinitz's theorem.

  The 4-polytopes with exactly seven vertices were completely characterised by Gr\"unbaum, in terms of their Gale diagrams: a complete list, together with the corresponding $f$-vectors, can be found in \cite[Figure 5]{Gru70}. Another listing of the Gale diagrams, without their $f$-vectors, can be found in \cite[Figure 6.3.3, 6.3.4]{Gru03}. The minimisers with seven vertices are the pyramid over $\TA(3)$ and the 2-fold pyramid over a pentagon, with $f$-vector (7,16,16,7). Every other 4-polytope with 7 vertices and at least 7 facets has $f_1\ge17$ and $f_2\ge17$.
  
  A 4-polytope with 8 vertices and 16 edges must be a prism, which has $d+2$ facets. There is no 4-polytope with 8 vertices and 17 edges \cite[Section 10.4]{Gru03}. There are exactly four 4-polytopes with 8 vertices and 18 edges, namely $\TA(4)$, $\Pm(3,1)$, $\WP$, and a pyramid over $\Sigma(3)$, all with $f$-vector $(8,18,17,7)$; see  \cite[Lem.~18]{PinUgoYos22}. Any 4-polytope with 8 vertices and at least 7 facets must satisfy
  $f_2=f_1+f_3-f_0\ge 18+7-8=17$,
  with strictly inequality unless $f_1=18$ and $f_3=7$. Thus these four polytopes are also the only minimisers of $f_2$ for 4-polytopes with 8 vertices and at least 7 facets. Any other 4-polytope with eight vertices and at least seven facets has more edges and 2-faces than these examples.
\end{proof}

  It is worth noting that the $f$-vectors of all 4-polytopes with up to nine vertices, together with the corresponding number of combinatorial types, can be found in \cite[Tables 6 and 7]{Fir20}. 
  
  We recall the smallest simple $d$-polytopes; see \cite[Lem.~2.19]{PinUgoYos16a}.
  
%
%
%
%
%

\begin{remark} The smallest vertex counts of simple $d$-polytopes are $d + 1$ (the $d$-simplex), $2d$ (the simplicial $d$-prism), and $3d-3$ (the polytope $T(2)\times T(d-2)$). These all have $d+2$ facets.
\label{rmk:simple-polytopes}
\end{remark}


\begin{theorem}\label{thm:strong-minimiser}
    Let $d\ge 3$ and $3 \le s \le d$, and let $P$ be a $d$-polytope with $d+s$ vertices and at least $d+3$ facets. Then $f_k(P) \ge \zeta_k(d+s,d)$ for each $k\in [1\ldots d-2]$.
\end{theorem}

\begin{proof}
    Proceed by induction on $d\ge 3$ for all $3 \le s \le d$.  The cases $d=3,4$ are covered in \cref{prop:small-cases}, so we assume that $d\ge 5$. \cref{rmk:simple-polytopes} gives that $P$ is not simple.	Let $v$ be a nonsimple vertex of $P$ with a maximum degree and let $F$ be a facet of $P$ that does not contain $v$ and has the maximum number $d+m$ of vertices where $0\le m\le d-1$.

\case $f_{0}(F)=f_{0}(P)-1$; that is, $m=s-1$.
\label{case:strong-minimiser-1}

The induction hypothesis on $d-1$ for $3\le s \le d-1$ yields that
\begin{equation}
\label{eq:pyramid}
\begin{aligned}
f_{k}(P)&=f_{k}(F)+f_{k-1}(F)\\
&\ge f_{k}(\Pm(s-1, d-s))+f_{k-1}(\Pm(s-1, d-s))\\
&=f_{k}(\Pm(s-1, d+1-s).
\end{aligned}
\end{equation}
	
If $s=d$ then $F$ has $2d-1=2(d-1)+1$ vertices and at least $(d-1)+3$ facets. In this case,  \cref{thm:2dplus1}  gives that $f_{k}(F)\ge f_{k}(\Pm(d-1))$. This combined with the relation $f_{k}(P)=f_{k}(F)+f_{k-1}(F)$  yields the result.  \qed

 \case $f_0(F)=d+m=d-1+m+1$, where $ m\in [2\ldots s-2]$. 
\label{case:strong-minimiser-2}
 \begin{subcases}
 \subcase \textit{The facet $F$ has at least $d+2=(d-1)+3$ facets.} 
 \label{subcase:2.1}

 
 Since $\deg_{P}(v)\ge d+1$, by \cref{prop:number-faces-outside-facet},  the number of $k$-faces containing at least one of the vertices outside $F$ is at least $ \phi_{k-1}(d+1, d-1)+ \sum_{i=2}^{s-m} \binom{d-i+1}{k}$. The induction hypothesis on $d-1$ yields that 
		\begin{equation}\label{eq:subcase2.1}
			\begin{aligned}
			f_k(P) & \ge f_k(F)+\phi_{k-1}(d+1, d-1)+ \sum_{i=2}^{s-m} \binom{d-i+1}{k}\\
			& \ge f_k(\Pm(m, d-m-1))+ \phi_{k-1}(d+1, d-1)+ \sum_{i=2}^{s-m} \binom{d-i+1}{k}         \\
			&  =\phi_k(d+m,d-1)+\binom{d-2}{k}-\binom{d-m-1}{k}+\phi_{k-1}(d+1, d-1)\\
			& \quad +\sum_{i=2}^{s-m} \binom{d-i+1}{k}    \\
			& =\binom{d}{k+1}+\binom{d-1}{k+1}-\binom{d-m-1}{k+1}+\binom{d-2}{k}-\binom{d-m-1}{k} \\
			&\quad +\binom{d}{k}+\binom{d-1}{k}-\binom{d-2}{k}+ \sum_{i=2}^{s-m} \binom{d-i+1}{k} \\
			& =\binom{d+1}{k+1}+\binom{d}{k+1}-\binom{d-m}{k+1}+ \sum_{i=2}^{s-m} \binom{d-i+1}{k} \\
			& = f_k(\Pm(s-1, d+1-s))+\binom{d+2-s}{k+1}-\binom{d-m}{k+1}-\binom{d-1}{k}\\
			&\quad + \sum_{i=2}^{s-m} \binom{d-i+1}{k}  \\
			& = f_k(\Pm(s-1, d+1-s))+\sum_{i=0}^{d+1-s}\binom{i}{k}-\sum_{i=0}^{d-m-1}\binom{i}{k}-\binom{d-1}{k}\\
			&\quad + \sum_{i=d-s+m+1}^{d-1} \binom{i}{k}  \\
			& = f_k(\Pm(s-1, d+1-s))+\sum_{i=0}^{d+1-s}\binom{i}{k}-\brac*{\sum_{i=0}^{d+1-s}\binom{i}{k}+\sum_{i=d+2-s}^{d-m-1}\binom{i}{k}}\\
			&\quad -\binom{d-1}{k}+ \sum_{i=d-s+m+1}^{d-1} \binom{i}{k}  \\
			& = f_k(\Pm(s-1, d+1-s))+\sum_{i=d+2-s}^{d-m-1}\brac*{\binom{m-1+i}{k} - \binom{i}{k}}
         \\  & \ge f_k(\Pm(s-1, d+1-s)).
			\end{aligned}
		\end{equation} 
The last inequality holds with equality if and only if $m=s-2$.
 

 \subcase \textit{The facet $F$ has $d+1=(d-1)+2$ facets and is not a simple polytope.} 
 \label{subcase:2.2}
  
 
 If $F$ is not a simple $(d-1)$-polytope, which occurs when either $f_0(F)\le 2d-3$ or $f_{0}(F)=2d-2$ and $F$ is not a simplicial $(d-1)$-prism  (\cref{rmk:simple-polytopes}), then, by \cref{lem:dplus2facets}, $F$ must be a pyramid whose  apex has degree at least $f_0(F)=d+m$ in $P$. Thus, $\deg_{P}(v)\ge d+m=d-1+m+1$. 

By \cref{prop:number-faces-outside-facet},  the number of $k$-faces containing at least one of the vertices outside $F$ is at least $ \phi_{k-1}(d+m, d-1)+\sum_{i=2}^{s-m}\binom{d-i+1}{k}$. Furthermore, $f_k(F)\ge \phi_{k}(d+m, d-1)$ by \cref{thm:at-most-2d}. Hence we have that
		\begin{equation}\label{eq:subcase2.2} 
			\begin{aligned}
			f_k(P) & \ge f_k(F)+\phi_{k-1}(d+m, d-1)+\sum_{i=2}^{s-m}\binom{d-i+1}{k}  \\
			&\ge \phi_{k}(d+m, d-1)+\phi_{k-1}(d+m, d-1)+\sum_{i=2}^{s-m}\binom{d-i+1}{k}  \\
			&  = \binom{d}{k+1}+\binom{d-1}{k+1}-\binom{d-m-1}{k+1}+\binom{d}{k}+\binom{d-1}{k}\\
			&\quad -\binom{d-m-1}{k} +\sum_{i=2}^{s-m}\binom{d-i+1}{k} \\
			& = f_k(\Pm(s-1, d+1-s))+\binom{d+2-s}{k+1}-\binom{d-m}{k+1} -\binom{d-1}{k}\\
			&\quad +\sum_{i=2}^{s-m}\binom{d-i+1}{k}  \\
			& = f_k(\Pm(s-1, d+1-s))+\sum_{i=0}^{d+1-s}\binom{i}{k}-\sum_{i=0}^{d-m-1}\binom{i}{k}
			-\binom{d-1}{k}\\
			&\quad +\sum_{i=d-s+m+1}^{d-1} \binom{i}{k} \\
			& = f_k(\Pm(s-1, d+1-s))+\sum_{i=d+2-s}^{d-m-1}\brac*{\binom{m-1+i}{k} - \binom{i}{k}} \\
                & \ge f_k(\Pm(s-1, d+1-s)).
			\end{aligned}
		\end{equation}

As in the previous scenario,  the  last inequality holds with equality if and only if $m=s-2$.
\begin{remark}[Number of facets containing the vertex $v$]\label{rmk:case2}
In case of equality, the vertex $v$ can have degree either $d+s-2$ or $d+s-1$ in $P$. Additionally, the number of $k-1$ faces in the vertex figure of $v$ must equal $\phi_{k-1}(d+s-2,d-1)$.
\begin{enumerate}[{\rm (i)}]
    \item If its degree is $d+s-2$, then its vertex figure must be a $(s-1,d-s)$-triplex, implying that $v$ is contained in exactly $d+1$ facets of $P$. 
    
    \item  If the degree of $v$ is $d+s-1$ and $s<d$, since 
\begin{equation*}
\phi_{k-1}(d+s-2,d-1)\le \phi_{k-1}(d+s-1,d-1)\le f_{k-1}(P/v),
\end{equation*}
 then \cref{thm:at-most-2d} gives that $P/v$ is a $(s,d-1-s)$-triplex, implying that $v$ is contained in exactly $d+1$ facets of $P$. From $\phi_{k-1}(d+s-1,d-1)=\phi_{k-1}(d+s-2,d-1)$, which amounts to  $\binom{d-s}{k}=\binom{d-s+1}{k}$, it follows that $\binom{d-s}{k-1}=0$. Hence, the equality in this case happens provided that $k\ge d-s+2$. 
    
\item Suppose that the degree of $v$ is $d+s-1$ and $s=d$. If $f_{d-2}(P/v)\ge d+2$, then \cref{thm:2dplus1} gives that $f_{k-1}(P/v)\ge \binom{d}{k}+2\binom{d-1}{k}-\binom{d-2}{k}>\phi_{k-1}(2d-2,d-1)$. If  
$f_{d-2}(P/v_1)=d+1$, then \cref{thm:2dplus1} gives that $f_{k-1}(P/v)\ge f_{k-1}((T^{d-1,d-1-(\floor{(d-1)/2}+2)}_{2})^*)>\phi_{k-1}(2d-2,d-1)$. Thus, equality does not arise in this case.
\end{enumerate}
    \label{rmk:vertex-dplus1-facets}
\end{remark}

\subcase \textit{The facet $F$ has $d+1=(d-1)+2$ facets and is a simple polytope.}
 \label{subcase:2.3}
 
\tinyskip 

This implies that $F$ has $2d-2$ vertices and is a simplicial $(d-1)$-prism. Label the vertices in a simplex $(d-2)$-face $R$ of $F$ as $w_1 ,\dots, w_{d-1}$ and the vertices in the other simplex $(d-2)$-face $R'$ of $F$ as $w'_1 ,\dots, w_{d-1}'$. Then $R\cap R'= \emptyset$. Additionally, label the two vertices outside $F$ as $v, v'$. Since $v$ has degree at least $d+1$ in $P$, there are at least $2d-1$ edges between $\set{v,v'}$ and $F$. Thus, there is a vertex in $F$, say $w_1'$, that is adjacent to both $v$ and $ v'$; see Figure~\ref{fig:1}.

    \begin{figure}
		\begin{center}
		\includegraphics[width=0.6\textwidth]{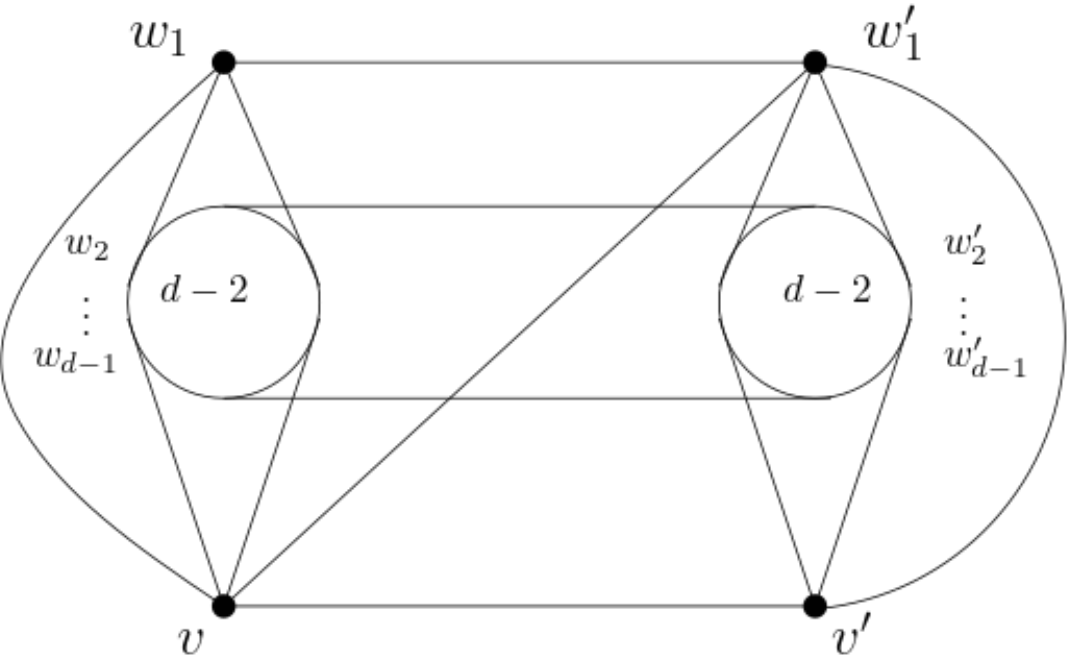}
		\end{center}
		\caption{$F$ is a prism.}
		\label{fig:1}
    \end{figure}

Consider the other facet $F_1$ that intersects $F$ at the simplicial $(d-2)$-prism $K_1=\conv\{w_i, w_i': i=2,\dots, d-1\}$. It follows that $f_{0}(F_1)=2d-3,2d-2$. 

First suppose that $f_{0}(F_{1})=2d-3$. Then $F_1$ is a pyramid over $K_1$ and thus, has $d+1$ $(d-2)$-faces. \green{The apex of $F_1$, namely $v$ or $v'$, has degree at least $2d-2$ in $P$ as it is adjacent to $w_1'$. 
The computation in \eqref{eq:subcase2.2} also applies here:}




\begin{equation}   
\label{eq:dplusd1-2dvertices} 
 \begin{aligned}
    f_k(P) & \ge f_k(F)+\phi_{k-1}(2d-2, d-1)+\binom{d-1}{k} \\
    & \ge \brac*{\binom{d}{k+1}+\binom{d-1}{k+1}-\binom{1}{k+1}}+\brac*{\binom{d}{k}+\binom{d-1}{k} -\binom{1}{k}}   \\
    & \quad +\binom{d-1}{k}\\
    & =\binom{d+1}{k+1}+\binom{d}{k+1}-\binom{2}{k+1}+\binom{d-1}{k} \\
    & =f_k(\Pm(d-1, 1)).
 \end{aligned}
\end{equation}
In case of equality in \eqref{eq:dplusd1-2dvertices}, we find that $v'$ cannot be the apex of $F_1$, as $v'$ is present in exactly $\binom{d-1}{k}$ $k$-faces not containing $v$. 
(We can actually apply \cref{eq:subcase2.2}, then in case of equality, it implies that $v'$ has to be simple in every facet not containing $v$, so $v'$ cannot be the apex of $F_1$.)
\green{Furthermore, by reasoning as in \cref{rmk:case2}(iii), since $f_{k-1}(P/v)=\phi_{k-1}(2d-2,d-1)$, we find that  $P/v$ must have $d+1$ $(d-2)$-faces and be a simplicial $(d-1)$-polytope. In particular, the vertex $v$ has degree exactly $2d-2$.}


Suppose that $F_1$ has $2d-2$ vertices. Then $F_{1}$ is not a pyramid. If $F_1$ has at least $d+2$ facets, then, since $w_{1}'$ is a nonsimple vertex outside $F_1$, we settle this by the calculation of \cref{eq:subcase2.1}. 
We can then assume $F_1$ has $d+1$ facets, in which case $F_1$ is a simplicial $(d-1)$-prism; see \cref{lem:dplus2facets}. Without loss of generality, we may assume that, within $F_1$, the vertex $v$ is adjacent to $w_2,\dots, w_{d-1}$ and $v'$ is adjacent to $w_2',\dots, w_{d-1}'$. It follows that $w_2,\dots, w_{d-1}$, $w_2',\dots, w_{d-1}'$ are simple vertices in $P$.


 As $F_1$ is a prism, the lines containing the edges $[w_i, w_i']$, $i=2,\dots,d-1$ and the line containing the edge $[v, v']$ have the property that every two of them are coplanar. Likewise as $F$ is a prism, the lines containing the edges $[w_i, w_i']$, $i=1,\dots, d-1$, have the property that every two of them are coplanar. It follows that $[w_1, w_1']$ is coplanar with $[v, v']$. Since the vertex $w_1'$ is adjacent to $v$ and  $v'$, and the vertex $w_1$ is adjacent to either $v$ or $v'$, we have a violation of convexity, therefore this case does not arise.
\qed
\end{subcases}
 
\case The facet $F$ has $d+1$ vertices. 
\label{case:strong-minimiser-3}

\tinyskip
	
There are $s-1$ vertices outside $F$, say $\set{u_{1},\ldots,u_{s-1}}$ where $u_{1}=v$. By \cref{prop:number-faces-outside-facet}, there is a sequence $F_1, \ldots, F_{s-1}$ of faces of $P$ such that each $F_i$ has dimension $d-i+1$ and contains $u_i$, but does not contain any $u_j$ with $j<i$. The number of $k$-faces in $P$ is bounded from below as follows: 
		\begin{equation}\label{eq:case3}
			\begin{aligned}
			f_k(P) & \ge f_{k}(F)+\phi_{k-1}(d+1, d-1)+\sum_{i=2}^{s-1}\binom{d-i+1}{k} \\
			& \ge \phi_k(d+1, d-1)+\phi_{k-1}(d+1, d-1)+\sum_{i=2}^{s-1}\binom{d-i+1}{k} \\
			& = \binom{d}{k+1}+\binom{d-1}{k+1}-\binom{d-2}{k+1}+\binom{d}{k}+\binom{d-1}{k}-\binom{d-2}{k}\\
			&\quad +\sum_{i=2}^{s-1}\binom{d-i+1}{k} \\
			& = f_k(\Pm(s-1, d+1-s))+\binom{d+2-s}{k+1}-\binom{d}{k+1}+\sum_{i=2}^{s-1}\binom{d-i+1}{k}\\
& = f_k(\Pm(s-1, d+1-s))+\sum_{i=0}^{d+1-s} \binom{i}{k}-\sum_{i=0}^{d-1} \binom{i}{k}+\sum_{i=2}^{s-1}\binom{d-i+1}{k} \\
& = f_k(\Pm(s-1, d+1-s))-\sum_{i=d+2-s}^{d-1}\binom{i}{k}+\sum_{i=2}^{s-1}\binom{d-i+1}{k} \\
			& = f_k(\Pm(s-1, d+1-s)).
			\end{aligned}
		\end{equation}
This settles this case. \qed

 \case The facet $F$ (and any other facet not containing $v$) is a simplex.
 \label{eq:case4}

\begin{subcases}
Let $v_{1}$ be a vertex not in $F$ other than $v$. By assumption, there must be a simplex facet $F'$ containing $v_{1}$ but not $v$. Let $v_{1},\ldots,v_{a}$ represent the vertices in $F'\setminus F$ and $v_{a+1},\ldots,v_{s}$ the vertices outside $F\cup F'$.  We consider two scenarios according the number $s-a$ of vertices outside $F \cup F'$. 	
	
	\tinyskip	
\subcase \textit{The inequality $s-a \ge 2$ holds.} 
 \label{subcase:4.1}
 
\tinyskip

By \cref{prop:number-faces-outside-facet}, there is a sequence $J_1,\ldots, J_{a}$ of faces in $F'$ such that each $J_i$ has dimension $d-1-i+1$ and contains $v_i$ but does not contain $v_j$ with $j<i$. Additionally, there is a sequence $F_{a+1},\ldots, F_{s}$ of faces in $P$ such that each $F_{\ell}$ has dimension $d-\ell+(a+1)$.  
 Consequently, we have the following. 

		\begin{equation*}\label{eq:subcase4.1}
			\begin{aligned}
			f_k(P) & \ge f_k(F)+ \#\; \text{of $k$-faces in $F'\setminus F$ containing a vertex in $\set{v_1,\dots, v_a}$}  \\
			& \quad +\#\; \text{of $k$-faces outside  $F\cup F'$ containing a vertex in $\set{v_{a+1},\dots,v_s}$} \\
			& \ge \binom{d}{k+1}+\brac*{\sum_{i=1}^{a} \binom{d-i}{k}}+\brac*{\phi_{k-1}(d+1,d-1)+\sum_{i=2}^{s-a}\binom{d-i+1}{k}} \\
			& = \binom{d}{k+1}+\sum_{i=1}^{a} \binom{d-i}{k}+\binom{d}{k}+\binom{d-1}{k}-\binom{d-2}{k}
		+\sum_{i=2}^{s-a}\binom{d-i+1}{k} \\
			& = f_k(\Pm(s-1, d+1-s))+\binom{d+2-s}{k+1}-\binom{d}{k+1}-\binom{d-2}{k} \\
			& \quad +\sum_{i=1}^{a} \binom{d-i}{k}+\sum_{i=2}^{s-a}\binom{d-i+1}{k}\\
			& = f_k(\Pm(s-1, d+1-s))+\sum_{i=0}^{d+1-s}\binom{i}{k}-\sum_{i=0}^{d-1}\binom{i}{k}-\binom{d-2}{k} \\
			& \quad +\sum_{i=1}^{a} \binom{d-i}{k}+\sum_{i=2}^{s-a}\binom{d-i+1}{k}\\
			& =  f_k(\Pm(s-1, d+1-s))-\sum_{i=d+2-s}^{d-1}\binom{i}{k}-\binom{d-2}{k} +\sum_{i=d-a}^{d-1}\binom{i}{k}\\
			&\quad +\sum_{i=d+a+1-s}^{d-1}\binom{i}{k}    \\
			& =  f_k(\Pm(s-1, d+1-s))-\sum_{i=d+2-s}^{d+a-s}\binom{i}{k}-\binom{d-2}{k} +\sum_{i=d-a}^{d-1}\binom{i}{k} \\
			& =  f_k(\Pm(s-1, d+1-s))+\brac*{\binom{d-1}{k}-\binom{d-2}{k}}\\
			& \quad +\sum_{i=d+2-s}^{d+a-s}\underbrace{\binom{i+s-a-2}{k}-\binom{i}{k}}_{\ge 0\;  \text{since $s-a-2\ge 0$}} \\
			& > f_k(\Pm(s-1, d+1-s)).
			\end{aligned}
		\end{equation*}

\subcase  \textit{The equality $s-a=1$ holds; that is, $v$ is the only vertex that is outside $F \cup F'$.} 
\label{subcase:4.2}

\tinyskip	

In this case, $\card (F \cap F') \ge 1$, and each vertex  in $F \cap F'$ is adjacent to every other vertex in $F \cup F'$, thereby implying that such a vertex has degree at least $d+s-2=\card (F \cup F')$ in $P$. By maximality, $\deg(v)\ge d+s-2$. Again, by \cref{prop:number-faces-outside-facet}, the number of $k$-faces containing a vertex outside $F$ is at least $\sum_{i=2}^{s}\binom{d-i+1}{k}+\phi_{k-1}(d+s-2, d-1)$. Hence, we have the following.
		\begin{equation}\label{eq:subcase4.2}
			\begin{aligned}
			f_k(P) 
 & \ge f_{k}(F)+\sum_{i=2}^{s}\binom{d-i+1}{k}+\phi_{k-1}(d+s-2, d-1) \\
 & \ge \binom{d}{k+1}+\sum_{i=2}^{s}\binom{d-i+1}{k}+\phi_{k-1}(d+s-2, d-1) \\
			&  =\binom{d}{k+1}+\binom{d}{k}+\binom{d-1}{k}-\binom{d+1-s}{k}+\sum_{i=2}^{s}\binom{d-i+1}{k} \\
			& =f_k(\Pm(s-1, d+1-s))-\binom{d+1}{k+1}+\binom{d+1-s}{k+1}+\binom{d}{k}\\
						&\quad +\sum_{i=2}^{s}\binom{d-i+1}{k} \\
			& =f_k(\Pm(s-1, d+1-s))-\sum_{i=1}^{s}\binom{d-i+1}{k}+\binom{d}{k}\\
						&\quad +\sum_{i=2}^{s}\binom{d-i+1}{k} \\			
			& =f_k(\Pm(s-1, d+1-s)).
			\end{aligned}
		\end{equation} 

We delve deeper into the degree of $v$ and the facets containing it. 

\begin{remark}\label{rmk:equality_case4} In case of equality in \eqref{eq:subcase4.2}, we have the following. 

(i) If $s\le d-1$ then $\card (F\cap F')\ge 2$. Since $v$ has degree at least $d+s-2$, it must be adjacent to a vertex $w \in F\cap F'$, as there are not $d+s-2$ vertices outside the intersection. Consequently, the vertex $w$ has degree $d+s-1$ in $P$, which implies that $v$ has degree $d+s-1$ in $P$.

(ii) Suppose that $s=d$. If $\deg(v)= 2d-1$, then the vertex figure $P/v$ has $2d-1$ vertices. Reasoning as with \cref{rmk:vertex-dplus1-facets}(iii), we see that this case does not arise.

Assume that $\deg(v)= 2d-2$. There is a facet that contains $v$ and has $2d-2$ vertices in total. Indeed, let $w$ be the unique vertex in $F\cap F'$. Then we get that  $\deg(w)=2d-2$ and that $w$ is not adjacent to $v$ (otherwise $\deg(w)=2d-1$). Furthermore, since $f_{k}(P/v)=\phi_{k}(2d-2,d-1)$, we find that  $P/v$ is a simplicial $(d-1)$-prism. Consider a simplicial $(d-2)$-prism $R$ in $P/v$. The facet $J$ of $P$ containing $v$ (and $R$) has at least $2(d-2)+1$ vertices. It follows that $J$ must also contain $w$; otherwise $J$ is a nonsimplex facet missing a vertex of largest degree.  That is, $f_{0}(J)=2d-2$.

In every scenario of the remark,  $v$ is contained in exactly $d+1$ facets of $P$. 
\label{rmk:case4-vertex-degree}
\end{remark}
\end{subcases} 
The proof of the theorem is complete. 
	\end{proof}

\subsection{Equality analysis}
\label{sec:equality}

\cref{thm:dplus2-vertices,thm:strong-minimiser} deal with the inequality part of our main theorem (\cref{thm:at-most-2d-refined}). We now characterise the equality cases. Let us highlight them.

\begin{remark} For the parameters $(d,s,m)$ where $d\ge 5$, $3\le s\le d$ and $0\le m\le s-1$, let $P$ be a $d$-polytope with $d+s$  vertices and at least $d+3$ facets, let $v$ be a nonsimple vertex with a largest degree in $P$, let $F$ be a facet not containing $v$ and  with a largest  possible number $d+m$ of vertices, and let $X$ be the set of vertices outside $F$. The proof of \cref{thm:strong-minimiser} establishes strict inequality if $2\le m\le s-3$. The possible equality cases can  occur only in the following settings.
 \begin{enumerate}[{\rm (i)}]
 \item \textbf{$m=s-1$ (\cref{thm:strong-minimiser}, \cref{case:strong-minimiser-1})}: The facet $F$ satisfies $f_k(F)=f_{k}(\Pm(s-1,d-s))$ for each $k\in [1\ldots d-2]$; see \eqref{eq:pyramid}. If $s=d$, then $\Pm(d-1,0)=\Pm(d-1)$.
 
 \item   \textbf{$m=s-2$ (\cref{thm:strong-minimiser}, Subcase~2.\ref{subcase:2.1})}: The facet $F$ has at least $d+2$ $(d-2)$-faces and satisfies $f_k(F)=f_{k}(\Pm(s-2,d-s+1))$ for $k\in [1\ldots d-2]$. Additionally,  the vertex $v$ has degree $d+1$ and its vertex figure is a $(2,d-3)$-triplex. The number of $k$-faces containing the other vertex in $X\setminus\set{v}$ but not $v$ is $\sum_{i=2}^{2}\binom{d-i+1}{k}=\binom{d-1}{k}$; see  \eqref{eq:subcase2.1}. 
 
\item   \textbf{$m=s-2$ (\cref{thm:strong-minimiser}, Subcase~2.\ref{subcase:2.2})}: The  vertex $v$ has degree $d+s-2$ and its vertex figure is a $(s-1,d-s)$-triplex. The facet $F$ has $d+1$ $(d-2)$-faces and is an  $(s-1,d-s)$-triplex. Additionally, the number of $k$-faces containing  the other vertex in $X\setminus\set{v}$ but not $v$ is $\sum_{i=2}^{2}\binom{d-i+1}{k}=\binom{d-1}{k}$; see  \eqref{eq:subcase2.2}.

\item \textbf{$m=s-2$ (\cref{thm:strong-minimiser}, Subcase~2.\ref{subcase:2.2})}: The vertex $v$ has degree $d+s-1$ and its vertex figure is a $(s,d-1-s)$-triplex with $s\le d-1$. The facet $F$ has $d+1$ $(d-2)$-faces and is an  $(s-1,d-s)$-triplex. Furthermore, $\phi_{k-1}(d+s-1,d-1)=\phi_{k-1}(d+s-2,d-1)$, which implies that $\binom{d-s}{k}=\binom{d-s+1}{k}$ or, equivalently that $k\ge d-s+2$; see \cref{rmk:vertex-dplus1-facets}(ii). Additionally, the number of $k$-faces containing  the other vertex in $X\setminus\set{v}$ but not $v$ is $\sum_{i=2}^{2}\binom{d-i+1}{k}=\binom{d-1}{k}$; see  \eqref{eq:subcase2.2}.


\item  \textbf{$m=d-2$ (\cref{thm:strong-minimiser}, Subcase~2.\ref{subcase:2.3})}: The facet $F$ is a simplicial $(d-1)$-prism,  there is a facet $F'$ which is a pyramid over a simplicial $(d-2)$-prism from $F$ and whose apex is $v$. The degree of $v$ in $P$ is exactly  $2d-2$ and its vertex figure is a simplicial $(d-1)$-prism. Additionally, the number of $k$-faces containing the vertex in $X\setminus\set{v}$ but not $v$ is $\binom{d-1}{k}$; see \eqref{eq:dplusd1-2dvertices}.

 \item \textbf{$m=1$ (\cref{thm:strong-minimiser}, \cref{case:strong-minimiser-3})}: The facet $F$ is a $(2,d-3)$-triplex, the vertex $v$ has degree $d+1$ and its vertex figure is a $(2,d-3)$-triplex, and the number of $k$-faces containing some vertex in $X\setminus\set{v}$  but not $v$
  is $\sum_{i=2}^{s-1}\binom{d-i+1}{k}$; see \eqref{eq:case3}.

\item \textbf{$m=0$ (\cref{thm:strong-minimiser}, Subcase~4.\ref{subcase:4.2})}: Every facet not containing a nonsimple vertex with the largest degree is a simplex. Moreover, there is a simplex facet $F'$ such that $F\cup F'$ contains $d+s-1$ vertices. If $s\le d-1$, then the vertex $v$ has degree  $d+s-1$ and its vertex figure is a $(s,d-1-s)$-triplex.  If $s= d$, then the vertex $v$ has degree  $2d-2$ and its vertex figure is a simplicial $(d-1)$-prism. Additionally, the number of $k$-faces containing some vertex in $X\setminus\set{v}$ but not $v$ 
  is $\sum_{i=2}^{s}\binom{d-i+1}{k}$; see \eqref{eq:subcase4.2} and \cref{rmk:case4-vertex-degree}.
 \end{enumerate}
\label{rmk:equality-cases}
\end{remark}

\begin{lemma} Let $d\ge 5$ and $3\le s\le d$ be given parameters, and let $P$ be a nonpyramidal  $d$-polytope $P$ with $d+s$ vertices and at least $d+3$ facets. If $f_k(P)=\zeta_k(d+s,d)$ for some $k\in[1\ldots d-2]$, then $P$ has exactly $d+3$ facets.
    \label{lem:facets-minimisers}
\end{lemma}   
\begin{proof} We are inspired by the reasoning in \cite[Sec.~4]{Xue21}. The scenarios where we get equality in the proof of \cref{thm:strong-minimiser} are listed in \cref{rmk:equality-cases}. In every case, the following hold:
\begin{enumerate}[{\rm (i)}]
    \item There is a nonsimple vertex $v_1$ (with the maximum degree) contained in exactly $d+1$ facets.
    \item  There is a facet $J$ with $d+m$ vertices not containing $v_1$ where $0\le m\le s-2$. 
    \item Let $Y:=\set{v_1,\ldots,v_{s-m}}$ be the set of vertices outside $J$.  The vertices outside $J$ other than $v_1$ are contained in exactly $\sum_{i=2}^{s-m}\binom{d-i+1}{k}$ $k$-faces of $P$. 
\end{enumerate}
 In  \cref{rmk:equality-cases}, the vertex $v_1$ is the vertex $v$, $Y=X$, and $J=F$. 

\begin{claim} The number of $k$-faces in $P$ containing the vertex $v_i$ and not $v_j$ with $i\ge 2$ and $j<i$ is precisely $\binom{d-i+1}{k}$. 
    \label{cl:facets-minimisers}
\end{claim}
\begin{claimproof} The minimum number of $k$-faces within a $(d-i+1)$-face that contains the vertex $v_{i}$ is attained when the vertex figure is a simplex, namely a $(d-i)$-simplex.  Since the $k$-faces in $P$ containing some vertex in $Y\setminus \set{v_1}$ is precisely $\sum_{i=2}^{s-m}\binom{d-i+1}{k}$, each vertex figure must be a minimiser. 
\end{claimproof}

The facets of $P$ can be partitioned into the facet $J$ and the sets $X_i$ of facets containing the vertex $v_i$ but no vertex $v_j$ with $j<i$. From the first paragraph we get that $\card X_1=d+1$. We prove that $\card X_2=1$ and $\card X_{\ell}=0$ for $\ell\in [3\ldots s-m]$.

Suppose that there are two facets $J_2$ and $J_2'$ in $X_{2}$.  Then there must exist a $k$-face $R$ of $J_{2}'$ that is not a face of $J_{2}$. Since there are already $\binom{d-1}{k}$ $k$-faces in $J_{2}$ containing $v_{2}$ (and not $v_{1}$), the number of $k$-faces in $P$ containing $v_{2}$ but not $v_{1}$ is greater  than $\binom{d-1}{k}$,     contradicting  \cref{cl:facets-minimisers}. Thus $\card X_{2}=1$.

The same idea proves that $\card X_{3}=0$. If there were a facet $J_{3}$ containing $ v_{3}$ but not $v_{1}$ or $v_{2}$, then the number of $k$-faces in $P$ containing $ v_{3}$ but not $ v_{1}$  or $ v_{2}$ would be at least the number of $k$-faces in $J_{3}$ containing $v_{3}$, namely $\binom{d-1}{k}$, which again contradicts \cref{cl:facets-minimisers}. Since we have the freedom to choose a different ordering of the    
vertices in $Y\setminus \set{v_1}$, any vertex vertex in $Y\setminus\set{v_{1},v_{2}}$  can be picked as $v_{3}$. Thus  $\card X_{\ell}=0$ for $\ell\in [3\ldots r]$, yielding that $P$ has $d+3$ facets.
\end{proof}

We write down a corollary of \cref{lem:facets-minimisers} to aid our  future discussion.

\green{\begin{corollary}\label{cor:case2_equality} Let $d\ge 5$ and $3\le s\le d$, let $P$ be a $d$-polytope with $d+s$ vertices and at least $d+3$ facets, let $v$ be a nonsimple vertex with the highest degree, and let $F$ be a facet not containing $v$ and with a largest possible number of vertices. In the equality condition of Subcases 2.\ref{subcase:2.2} and  2.\ref{subcase:2.3} of \cref{thm:strong-minimiser}, the facet $F$  has $d+s-2$ vertices and  $d+1$ $(d-2)$-faces. Moreover, for  the vertex $v'\ne v$ outside $F$,  the following hold:
  \begin{enumerate}[{\rm (i)}]
\item    There is a unique facet containing both $v$ and $v'$, which intersects $F$ at a $(d-3)$-face.
\item    There is a unique facet containing $v'$ but not $v$, which intersects $F$ at a $(d-2)$-face.  
\item    There are either $d-2$ or $d-1$ facets containing both $v$ and $v'$ that intersect with $F$ at a $(d-2)$-face. 
  \end{enumerate}
\end{corollary}}
\begin{proof} \green{In the notation of \cref{lem:facets-minimisers}, we assume that the facet $F$ here is the facet $J$ there, the vertex $v$ is the vertex $v_{1}$ there, and the vertex $v'$ is the vertex $v_{2}$ there.} 

\green{(i) Since $F$ contains exactly d+1 $(d-2)$-faces and, by \cref{lem:facets-minimisers} $P$ contains exactly $d+3$ $(d-1)$-faces, there must exist a unique facet in $P$ consisting of the convex hull of a $(d-3)$-face of $F$ and the vertices $v$ and $v'$.} 
 
\green{(ii) The proof of Claim 1 in \cref{lem:facets-minimisers} shows that the set $X_{2}$ of facets containing $v'$ but not $v$ has cardinality one.}

\green{(iii) Parts (i) and (ii) count two facets in $P$ containing $v'$. The remaining facets containing $v'$ must also contain $v$ and a $(d-2)$-face of $F$. If $v'$ is simple in $P$, there are $d-2$ such facets; otherwise there are $d-1$.}     
\end{proof}

\green{\begin{theorem}[Equality] Let  $d\ge s\ge3$  be parameters and let $P$ be a $d$-polytope  with $d+s$ vertices and at least $d+3$ facets. If, for some $k\in [1\ldots d-2]$,  
\begin{equation*}
f_{k}(P)=f_{k}(\Pm(s-1,d+1-s))=\zeta_k(d+s,d),
\end{equation*} then the following statements hold: 
\begin{enumerate} [{\rm (i)}]
\item If $s=3$, then $P$ is either the $(d-2)$-fold pyramid over a pentagon or the $(d-3)$-fold pyramid over $\TA(3)$.   
\item If $5\le s\le{d-1}$ and $k\le d-s+1$, then $P$ is either the $(d-s+1)$-fold pyramid over an $(s-1)$-pentasm or the $(d-s)$-fold pyramid over $\TA(s)$. 
\item If $5\le s\le{d-1}$ and $d-s+2\le k\le d-2$, then $P$ is either the $(d-s+1)$-fold pyramid over an $(s-1)$-pentasm, the $(d-s)$-fold pyramid over $\TA(s)$, or the $(d-s-1)$-fold pyramid over $Z(s+1)$. 
\item If $s=4$ and either $k\le{ d-3}$ or $d=4$ and $k=2$, then $P$ is either the $(d-3)$-fold pyramid over an 3-pentasm, the $(d-3)$-fold pyramid over $\Sigma(3)$, the $(d-4)$-fold pyramid over $\TA(4)$, or the $(d-4)$-fold pyramid over $\WP$. 
\item   
 If $s=4$, $k=d-2$ and $d\ge5$, then $P$ is either the $(d-3)$-fold pyramid over an 3-pentasm, the $(d-3)$-fold pyramid over $\Sigma(3)$, the $(d-4)$-fold pyramid over $\TA(4)$, the $(d-4)$-fold pyramid over $\WP$  or  the $(d-5)$-fold pyramid over $Z(5)$. 
\item If $s=d$, then $P$  is either $\TA(d)$ or the pyramid over a $(d-1)$-pentasm.   
\end{enumerate}
\label{thm:at-most-2d-refined-equality}
\end{theorem}}

\begin{proof}
    
We proceed by induction on the dimension. The induction base is $d=3$ and 4, which is covered by \cref{basecase}. 
    
Let $P$ be a $d$-polytope as in the hypothesis.  Suppose now the theorem holds for dimensions $5, 6, \dots, d-1$. As before, let $v$ be a nonsimple vertex of $P$ with maximum degree and let $F$ be a facet of $P$ that does not contain $v$ and has a maximum number $d+m$ of vertices (amongst such facets). We consider each of the scenarios listed in \cref{rmk:equality-cases}.
 
\tinyskip 
   
\noindent \green{{\textbf{\cref{rmk:equality-cases}(i)}}: The facet $F$ has $d+s-1$ vertices. If $s \le d-1$, by induction on $d-1$, the facet is one of the minimisers. Thus, $P$ is a pyramid over $F$ and has the desired form. If $s=d$, by \cref{thm:strong-minimiser}, \cref{case:strong-minimiser-1} (and \cite{PinYos22}), the facet $F$ is a $(d-1)$-pentasm. Hence, $P$ is a pyramid over the pentasm; that is, a $\Pm(d-1, 1)$.}

\green{We now consider the scenarios in \cref{rmk:equality-cases}(ii)--(v). Here, the facet $F$ has $d+s-2$ vertices with $s\ge 4$. (If $s=3$ then $F$ has $d+1$ vertices, a scenario considered in \cref{rmk:equality-cases}(vi).) Denote by $v'$ the  vertex other than $v$ and outside $F$. In \cref{rmk:equality-cases}(ii)--(v), note that the vertex $v'$ is simple in the unique facet that contains it but does not contain  $v$.}

 \tinyskip
 
\noindent \green{{\textbf{\cref{rmk:equality-cases}(ii)}}: The facet $F$ has at least $d+2$ $(d-2)$-faces and satisfies $f_{k}(F)=f_{k}(\Pm(s-2,d-s+1))$. By the induction hypothesis, $F$ is either $\Pm(s-2, d-s+1)$, $\TA(s-1, d-s)$, $Z(s, d-s-1)$, $\pyr_{d-5}(\WP)$, or $\pyr_{d-4}(\Sigma(3))$.}

\green{First suppose that $F$ is not a pyramid. For $d\ge 5$, this means that the facet $F$ could  be only $\TA(d-1)$, $Z(d-1)$, or $\WP$. \cref{rmk:structure-TA,rmk:structure-Z}  assert that $\TA(d-1)$ and $Z(d-1)$ both have a vertex with degree $2d-4$, which results in a degree at least $2d-3$ in $P$. Since the vertex $v$ has degree $d+1$, so does every nonsimple vertex in $P$. Thus, we find that $d\le 4$, which has been already covered. }
     

  
\green{It remains to consider the scenario where $F$ is $\WP$. We denote by $v_1, v_2, v_3, v_4$ the four nonsimple vertices of $F$; see \cref{fig:polytopes}(b). 
Since $v'$ is simple in the unique facet containing it but not $v$, this facet must be the simplex with vertices $v',v_{1},v_{2},v_{3},v_{4}$. 
Additionally, because the degree of $v$ is six, every nonsimple vertex in $P$ must have degree six. Thus, the vertices $v_{1},v_{2},v_{3},v_{4}$  have each exactly one neighbour outside $F$, which is  $v'$. As a consequence, the neighbours of $v$ in $P$ other than $v'$ must be the four vertices in $F$ other than $v_{1},v_{2},v_{3},v_{4}$, ensuring that the the degree of $v$ in $P$ is five, a contradiction. }
%
%
%
%
     
\green{Assume that $F$ is a pyramid. Since $f_{0}(F)\ge d+2$, its apex has at least $d+1$ neighbours in $F$ and, therefore has degree at least $d+2$ in $P$, contradicting  the degree of $v$.}
   
\tinyskip 
 
\noindent \green{{\textbf{\cref{rmk:equality-cases}}(iii)}: The facet 
    $F$ has $d+1$ $(d-2)$-faces and is a $(s-1, d-s)$-triplex. Then $d-s\ge 1$.  Moreover,  the vertex $v$ has degree $d+s-2$.}
    
\green{Suppose that $d-s \ge 2$. There are at least two apexes $x_1$ and $x_{2}$ in $F$. For $i=1,2$, consider the other facet $F_i$ in $P$ intersecting  $F$ at a $(d-2)$-face that does not contain $x_i$. Since $P$ is not a pyramid, each $F_i$ must be pyramid with apex  either $v$ or $v'$ (but not both). Since $s\ge 4$, each facet $F_i$ is not a simplex. Therefore, as $v'$ is simple in each of the facets containing it, for each $i=1,2$, $v'\not\in F_{i}$. Thus, $v\in F_i$ for each $i=1,2$.  It follows that $v$ is adjacent to every other vertex in the polytope; that is, the degree of $v$ is $d+s-1$, contradicting our assumption. }

\green{Assume that $d-s=1$. Then the facet $F$ is a pyramid over a simplicial $(d-2)$-prism $R$. Then $R$ is the convex hull of two $(d-3)$-simplices $K$ and $K'$. Denote by $x$ the apex of $F$.}

\green{Consider a facet $F'$ in $P$ containing $R$ but not $x$. The facet $F'$ is a pyramid with apex either $v$ or $v'$; otherwise, $f_0(F')=d+s-1$ and $P$ would be pyramid whose apex $x$ has degree $2d-2$, which is larger than the degree of $v$, contradicting our assumption. It follows that $v \in F'$, as $v'$ is simple in each facet that contains it but does not contain  $v$. Because $\deg_{P}(v)=2d-3$, $v$ cannot be adjacent to $x$.}
	
\green{It is true that $\deg_{P}(v') <\deg_{P}(v)$; otherwise, we may pick $v'\notin F$ as our nonsimple vertex with the maximum degree. Additionally, $v$ would need to be simple the facet that contains it but does not contain $v'$. However, $v$ is not simple in $F'$.}
		
\green{    Consider a facet $J'$ in $P$ containing $v'$ but not $v$. Then $J'$ must be a pyramid with apex  $v'$ and intersecting $F$ at a ridge. Since $v'$ is simple in $J'$, $J'$ must be a simplex. Without loss of generality, we may assume  that $J'=\conv(\{ v',x\}\cup K')$. Since $\deg(v') < \deg(v)=2d-3$, we find that $v'$ is not adjacent to at least two vertices in $K$. Finally, we consider another facet $J$ in $P$ containing the $(d-2)$-simplex $\conv \{ x\}\cup K$. Since $v'$ is not adjacent to two vertices in $K$, the degree of $v'$ in $J$ would be at most $d-2$, a contradiction. Thus, $v'\notin J$ and, therefore $v\in J$. On the other hand, because $v$ is not adjacent to $x$, the vertex $v\notin J$, a contradiction.}   
		
 
\tinyskip
 
\noindent \green{{\textbf{\cref{rmk:equality-cases}}(iv)}: The facet 
    $F$ has $d+1$ $(d-2)$-faces and is a $(s-1, d-s)$-triplex.   Denote by $x_1,\ldots, x_{d-s}$ the apexes in $F$, where $d-s\ge 1$.  Moreover,  the vertex $v$ has degree $d+s-1$.}

\green{For $i=1,\ldots, d-s$, consider the facet $F_{i}$ in $P$ intersecting $F$ at the $(d-2)$-face $F\setminus \{x_{i}\}$. If $F_{i}$ contained both $v$ and $v'$, then the polytope would be a pyramid, a case already been covered. Thus, $F_{i}$ is a pyramid with apex $v$, since $f_{0}(F_{i})\ge d+2$ and $v'$ is simple the facet that contains it but does not contain $v$. It follows that, for each $i=1,\ldots, d-s$, $F_{i}$ contains $v$ but not $v'$. Additionally, consider a facet $J'$ in $P$ containing $v'$ but not $v$. Then $J'$ must be a pyramid with apex  $v'$ and intersecting $F$ at a ridge. Since $v'$ is simple in $J'$, the facet $J'$ must be a simplex.}
    
\green{Suppose that $d-s \ge 3$.  Since $F$ does not contain $v'$, there are at least four facets in $P$ not containing $v'$. However, by \cref{lem:facets-minimisers}, the polytope has exactly $d+3$ facets, leaving  at most $d-1$ facets containing $v'$, which is not possible.}

\green{Suppose $d-s=2$. Then $F$ is a two fold pyramid over a simplicial $(d-3)$-prism $R$. Then $R$ is a convex hull of two $(d-4)$-simplices $K$ and $K'$.  Without loss of generality, we assume that $J'=\conv(\{v',x_{1},x_{2}\}\cup K)$.} 

%

\green{By \cref{lem:facets-minimisers}, the polytope has exactly $d+3$ facets. The vertex $v'$ is not present in precisely three of those  facets, namely $F_{1}$, $F_{2}$, and $F$. Thus, $v'$ is simple in $P$, and so is not adjacent to any of the vertices in $K'$. For $i=1,2$, there is a facet $S_i$ in $P$ containing $v'$ but not $x_i$. Since $v'$ is simple in $P$, each $S_{i}$ must intersect $F$ at a $(d-3)$-face  that contains $K$ and the other apex, and so must contain $v$. This violates \cref{cor:case2_equality}(i). }

    
\green{Assume that $d-s=1$. Then $F$ is a pyramid over a simplicial $(d-2)$-prism $R$, and so $f_{0}(F)=2d-3$. The face $R$ is the convex hull of two $(d-3)$-simplices $K:=\conv\{u_{1},\ldots,u_{d-2}\}$ and $K':=\conv\{u'_{1},\ldots,u'_{d-2}\}$ so that the vertices $u_{i}$ and $u_{i}'$ are adjacent. Without loss of generality, we assume that $J'=\conv(\{v',x_{1}\}\cup K)$. The facet $F_{1}$ is a pyramid with apex $v$ over $R$. This means that the vertices in $K$ has degree exactly $d+1$ and the vertices in $K'$ has degree at most $d+1$. }

\green{For $i=1,\ldots,d-2 $, consider the other facet $J_{i}$ in $P$ containing the $(d-2)$-face $I_{i}$ of $F$ missing the edge $u_{i}u_{i}'$. 
By \cref{lem:facets-minimisers}, the polytope has exactly $d+3$ facets. The vertex $v'$ is not present in two of those  facets, namely $F_{1}$ and $F$. As a consequence, $v'$ needs to be present in at least $d-3$ of such facets $J_{i}$, say in those for $i\in [1...d-3]$. 
 Since $v'$ is simple in the facet that contains it but does not contain $v$, the vertex $v$ must be present in all the facets $J_{i}$. Hence, $f_{0}(J_{i})=f_{0}(F)=2d-3$, for  $i\in [1\ldots d-3]$.}

  
\green{For  $i\in [1...d-3]$,  each facet $J_{i}$ is not a pyramid or a simple polytope, and so  has at least $d+2$ $(d-2)$-faces. Indeed, if $J_{i}$ were a pyramid, then either $v$ or $x_{i}$ would be the apex. In case of $v$ being the apex,  $I_{i}\cup \{v'\}$ would be the base but $I_{i}$ is already a $(d-2)$-face; a similar argument establishes that $x_{1}$ cannot be the apex of $J_{i}$. It is also straightforward that $J_{i}$ is not a simple polytope. From \cref{lem:dplus2facets}, it now follows that $J_{i}$ cannot have $d+1$ $(d-2)$-faces. }
    


    
 
\green{For each $i\in [1\ldots d-3]$, since the vertex $u_{i}$ (which is outside $J_{i}$) has degree $d+1$, $J_{i}$ is a largest facet with at least $d+2$ $(d-2)$-faces, and $f_{k}(P)=f_{k}(\Pm(s-1,d+1-s))$ for some $k\in [1\ldots d-2]$, the computation \eqref{eq:subcase2.1} in Subcase \ref{case:strong-minimiser-2}.\ref{subcase:2.1}   applies and gives that $f_{k}(J_{i})=f_{k}(\Pm(s-2,d+1-s))$. In this scenario, the induction hypothesis on $d-1$ ensures that $J_{i}$ is one of the polytopes in Parts (ii)--(v) of the theorem. Since $d\ge 5$, $J_{i}$ is not a pyramid, and $v$ (and $x_{1}$) is adjacent to every other vertex in $J_{i}$, $J_{i}$ can be only $Z(d-1)$. According to the structure of $Z(d-1)$ (\cref{rmk:structure-Z}),  for $i\in [1\ldots d-3]$, the $d-3$ vertices in $K$ and $J_{i}$ has degree $d$ in $J_{i}=Z(d-1)$, so  $v'$ and the $d-3$ vertices in $K'$ and $J_{i}$ must be simple in $Z(d-1)$. Since $d-3\ge 2$, this implies that $v'$ is not adjacent to any vertex in $K'$. This gives the  full graph of $P$, which consists of the union of the graphs of $F$,  $F_{1}$, and $J'$. In particular, 
\begin{itemize}
\item $f_{0}(P)=2d-1$,
\item The $d-1$ vertices in $K'\cup \set{v'}$ are simple in $P$,
\item The $d-2$ vertices in $K$ has degree $d+1$ in $P$, 
\item  the vertices $v$ and $x_{1}$ have degree $2d-2$ in $P$. 
\end{itemize}
See \cref{fig:polytopes}(d) and \cref{rmk:structure-Z}.}




\green{We now establish the vertex-facet incidence of the polytope. For each $i\in [1\ldots d-2]$, let us focus on one $J_i$ and the other facets that intersects $J_i$ at a ridge (\cref{rmk:structure-Z}). Since $J_i$ has exactly $d+2$ ridges of $P$, every other facet of $P$ intersect $J_{i}$ at a $(d-2)$-face of $J_{i}$ . Because the degree of $u_{i}$ in $P$ is $d+1$ and that of $u_{i}'$ is $d$, the other facet intersecting $J_i$ at $Z(d-2)$ must contain both $u_i$ and $ u_i'$ and, thus must be one of the facets $J_{\ell}$. 
The $d-3$ copies of $Z(d-2)$ in $J_{i}$ plus $J_{i}$ result in $d-2$ facets $Z(d-1)$ in $P$ and establish the vertex-facet incidence among them. Similarly, a facet that intersects  $J_i$ at a one of the two copies of $M(d-3,1)$ in $J_{i}$ must contain both $u_i, u_i'$. This gives two facets $M(d-2,1)$ of $P$ and the vertex-facet incidence among them. There are three simplex $(d-2)$-faces in $J_i$ (\cref{rmk:structure-Z}): 
\begin{align*}
	&\conv (\set{ v', v}\cup \set{u_j\in K: j\ne i}),\\ 
	&\conv (\set{ v', x_{1}}\cup \set{u_j\in K: j\ne i}),\\
	&\conv (\set{ v, x_{1}}\cup \set{u_j'\in K': j\ne i}).
\end{align*}
Since $v'$ is simple in $P$, one can see that the other facet containing one of such $(d-2)$-faces will be, respectively  
\begin{align*}
	&\conv (\set{ v', v}\cup K),\\ 
	&\conv (\set{ v', x_{1}}\cup K),\\
	&\conv (\set{ v, x_{1}}\cup K').
\end{align*}
The complete vertex-facet incidence of $P$ is established and, so the polytope is $Z(d)$. }  

 \tinyskip
  
\noindent \green{{\textbf{\cref{rmk:equality-cases}(v)}}: The facet $F$ is a simplicial $(d-1)$-prism and the polytope $P$ has $2d$ vertices.  Moreover,  the vertex $v$ has degree exactly $2d-2$.} 
        
\green{Following the proof of Subcase~\ref{case:strong-minimiser-2}.\ref{subcase:2.3} of \cref{thm:strong-minimiser}, we label the vertices of one $(d-2)$-simplex $R$ of $F$ by $w_1 ,\dots, w_{d-1}$ and the vertices of the other $(d-2)$-simplex $R'$ of $F$ by $w'_1 ,\dots, w_{d-1}'$, with $R\cap R'= \emptyset$. Additionally, let $v'$ be the vertex outside $F$ and distinct from $v$.} 
Let $S$ be the facet of $P$ containing $v'$ but not $v$. Then $S$ must be a pyramid (with apex $v'$) over a $(d-2)$-face contained in $F$. Since $v'$ has to be simple in $S$, it follows that $S$ is a simplex and either $S=\conv(R'\cup \{v'\})$ or $S=\conv(R\cup \{v'\})$. 


Only one vertex of $P$ is not adjacent to $v$, so
we can assume without loss of generality that $w'_{1}$ is adjacent to both $v$ and $v'$. 

\green{Now consider the other facet $F_1$ that intersects $F$ at the simplicial $(d-2)$-prism $K_1=\conv\{w_i, w_i': i=2,\dots, d-1\}$. It follows that $f_{0}(F_1)=2d-3$ or $2d-2$. }

\green{\textbf{First suppose that $f_0(F_1)=2d-3$.} Then the facet $F_{1}$ is a pyramid with apex $v$ over the simplicial $(d-2)$-prism $\conv((R\setminus \set{w_{1}})\cup (R'\setminus \set{w'_{1}}))$ of $F$. It follows that $v$ is adjacent to every vertex in $P$ except for $w_{1}$. Thus, $w_{1}$ is adjacent to $v'$.}

\green{It follows that $\deg(v')<2d-2$. Otherwise  $v'$ would also be a vertex with maximum degree, so the role of $v$ and $v'$ can be swapped. As $F_1$ is a facet containing $v$ but not $v'$ where $v$ is not simple, we would find a contradiction to the assumption that the unique facet containing $v'$ and not $v$ must be simple.
}

\green{First suppose that $S=\conv(R'\cup \{v'\})$. Since $\deg_P(v')<2d-2$, the vertex $v'$ is not adjacent to at least two vertices from $w_2, \dots, w_{d-1}$. Consider the other facet $J$ of $P$ that intersects $F$ at $R$. } 
Since $v'$ is not adjacent to at least two vertices in $R$, the vertex $v'$ would have at most $d-2$ neighbours in $J$, so $v' \notin J$. Then, since $v$ is not adjacent to $w_1$, the vertex $v$ would have at most $d-2$ neighbours in $J$, so $v \notin J$. This is absurd and this case does not arise.


\green{Now suppose that $S=\conv(R\cup \{v'\})$. Consider a facet $G$ that intersects $F$ at the simplicial $(d-2)$-prism $\conv\{ w_i, w_i': i=1,\ldots, d-2 \}$. 
The facet $G$  cannot be a pyramid, as $v$ is not adjacent to $w_{1}\in G$ and $\deg_P(v')<2d-2$. Thus, $v,v'\in G$. Furthermore, $G$ cannot be a simple polytope, since the vertex $w_{j}$  (with $j\ne 1,d-1$) is not simple in $G$ as it is adjacent to both $v$ and $v'$.  From \cref{lem:dplus2facets}, it now follows that $G$ has $d+2$ $(d-2)$-faces. Since $w_{d-1}$ is not a simple vertex and since $f_{k}(P)=f_{k}(\Pm(s-1,d+1-s))$ for some $k\in [1\ldots d-2]$, the computation \eqref{eq:subcase2.1} in Subcase \ref{case:strong-minimiser-2}.\ref{subcase:2.1}   applies and gives that $f_{k}(G)=f_{k}(\Pm(s-2,d+1-s))$ for some $k\in [2\ldots d-1]$.  By induction, the facet $G$ must be one of the nonpyramidal minimisers with $2d-2$ vertices, namely $\TA(d-1)$ or $\WP$. Since the $d-1$ vertices $w_1', v', w_2, \dots, w_{d-2}$ have degree $d$ in $G$, the facet $G$ cannot be $\TA(d-1)$ (\cref{rmk:structure-TA}). When $d=5$, the degree of $v$ in $G$ is 6, so $G$ cannot be $\WP$ (\cref{fig:polytopes} ). Hence, this case does not arise.}


\green{\textbf{It remains to consider the case $f_0(F_1)=2d-2$.} The proof of Subcase~\ref{case:strong-minimiser-2}.\ref{subcase:2.3} gives that $F_1$ has at least $d+2$ facets, as it is not a pyramid or a simple polytope. As $w_1'$ is a nonsimple vertex outside $F_1$, the facet $F_{1}$ is a largest facet with at least $d+2$ $(d-2)$-faces, and $f_{k}(P)=f_{k}(\Pm(s-1,d+1-s))$ for some $k\in [1\ldots d-2]$, the computation \eqref{eq:subcase2.1} in Subcase \ref{case:strong-minimiser-2}.\ref{subcase:2.1} applies and gives that $f_{k}(F_{1})=f_{k}(\Pm(s-2,d+1-s))$. In this scenario, the induction hypothesis on $d-1$ ensures $F_1$ is one of the nonpyramidal minimisers with $2d-2$ vertices, namely $\TA(d-1)$ or $\WP$. Since $\deg_{P}(v)=2d-2$, the degree of $v$ in $F_{1}$ is at least $2d-4$. Thus, when $d=5$, the facet $F_1$ cannot be $\WP$ (\cref{fig:polytopes}). Consequently, $F_1$ can be only $\TA(d-1)$. By the structure of $\TA(d-1)$ (\cref{rmk:structure-TA}), the degree of $v$ in $F_{1}$ is $2d-4$, so $v$ must be adjacent to $w_{1}$, and $v'$ must be simple in $F_{1}$. Moreover, without loss of generality, we may assume that $w_{d-1}$ and $w_{d-1}'$ are simple in $F_1$, and therefore simple in $P$. It follows that $v$ is not adjacent to either $w_{d-1}$ or $w_{d-1}'$. 
By the structure of $\TA(d-1)$ (\cref{rmk:structure-TA}), either $w_i$ or $w_i'$ is not simple in $F_{1}$ (for each $i\ne 1,d-1$) and therefore not simple in $P$, respectively. See \cref{fig:polytopes}.
}

\green{For $i=2,\ldots,d-2 $, consider the other facet $F_{i}$ in $P$ containing the $(d-2)$-face of $F$ missing the edge $w_{i}w_{i}'$. None of these facets can be pyramids, as $v$ cannot be an apex; recall that the vertices $w_{d-1}$ and $w_{d-1}'$ are in each such $F_{i}$ and $v$ is not adjacent to one of them. None of these facets are simple polytopes, as either $w_j$ or $w_j'$ (for $j\ne i,d-1$) is nonsimple in $F_{i}$. Thus, each $F_{i}$ has $d+2$ facets and $2d-2$ vertices. The same reasoning used with $F_1$ ensures the facets $F_i$ (for $i=2,\ldots,d-2$) are all $\TA(d-1)$. Hence, this gives $d-2$ copies of $\TA(d-1)$ as facets of $P$.} 
 
\green{ Since $w_1'$ is adjacent to $v'$ and $v$ is simple in each of the facets $F_{i}$, the graph of $P$ is now clear: the vertex $v'$ is adjacent to the vertices $v$, $w_1', \dots, w_{d-1}'$; the vertex $v$ is adjacent to the vertices $v'$, $w_1',\ldots, w_{d-2}'$, and $w_1,\ldots, w_{d-1}$.}

\green{ With the graph of $P$ in place and using \cref{cor:case2_equality}, we find that the unique facet containing $v'$ but not $v$ must be the simplex $\conv(R'\cup \set{v'})$. Since $w_{d-1}$ is simple in $P$, there is a further  simplex facet $\conv(R\cup \{ v\})$. Furthermore, the unique facet containing both $v$ and $v'$ and intersecting $F$ at a $(d-3)$-face must be the simplex $\conv(\{v,v'\}\cup \{w_i': i=1,\dots, d-2\})$. Hence, this gives three simplex facets of $P$.  }

\green{Since $v'$ is simple in $P$ and there are already $d$ facets containing $v'$, the other facet $F_{d-1}$ in $P$ containing the $(d-2)$-face of $F$ missing the edge $w_{d-1}w_{d-1}'$, does not contain $v'$. Thus $F_{d-1}$ must contain $v$, implying that $F_{d-1}$ is a $M(d-2, 1)$.We have now produced all the $d+3$ facets of $P$. These facets now give the vertex-facet incidence of the polytope, implying that $P$ is $\TA(d)$.}

\tinyskip

\noindent \green{{\textbf{\cref{rmk:equality-cases}(vi)}}: The facet 
    $F$ has $d+1$ vertices is a $(2,d-3)$-triplex.  Moreover,  the vertex $v$ has degree exactly $d+1$. Every vertex outside $F$ other than $v$ is simple the facet of $P$ containing it but not $v$. }
 
    
\green{As in \cref{case:strong-minimiser-3} of \cref{thm:strong-minimiser}, we denote  the $s-1$ vertices outside $F$ by $u_1,\ldots,u_{s-1}$ where $u_{1}=v$.  Denote the apexes of $F$ by $x_1,\ldots,x_{d-3}$.  For $i\in [1\ldots d-3]$, consider the facet $F_i$ intersecting $F$ at the $(d-2)$-face $F\setminus \set{x_{i}}$. It follows that $F_i$ has exactly $d+1$ vertices, and so $F_i$ is a pyramid. No vertex in $\set{u_2,\ldots,u_{s-1}}$ can be the apex of $F_{i}$, since they need to be simple in the facet containing them but not $v$. Thus, $v$ is the apex of $F_i$ for each $i$. It follows that $v$ is adjacent to all the vertices of $F$, and since the vertices outside a facet form a connected subgraph \cite[Thm.~4.1.6]{Pin24}, the degree of $v$ in $P$ would be at least $d+2$, contradicting our assumption.}	

\tinyskip
		
\noindent \green{{\textbf{\cref{rmk:equality-cases}(vii)}}: The facet 
    $F$ is a simplex, and so is every facet not containing a vertex of the largest degree.  Every vertex outside $F$ other than $v$ is simple in every facet of $P$ containing it but not $v$.} There is another simplex facet $F'$ such that $F\cup F'$ contains $d+s-1$ vertices.

\green{Suppose that  $s\le d-1$. By \cref{rmk:equality_case4}(i),  $\deg(v)=d+s-1$, and so is every vertex $F \cap F'$. The vertex figure $P/v$ of $v$ is $M(s, d-s-1)$. If $d-s-1\ge 1$, $P/v$ is a pyramid. Then the polytope is a pyramid: consider the facet containing $v$ but not the  vertex in $P$ corresponding to the apex of $P/v$; this facet necessarily contains $d+s-1$ vertices. The case of $P$ being a pyramid has already been covered. In case $d-s-1=0$,  $P/v$ is a simplicial $(d-1)$-prism and $\card (F\cap F')=2$. A vertex $w$ in $F\cap F'$ has the same degree $2d-2$ as $v$. Denote by $w'$ a vertex in $P/v$ corresponding to $w$. In $P/v$ there is a face that is a simplicial $(d-2)$-prism $R$ and does not contain $w'$. The facet $J$ in $P$ containing $v$ but not $w$ has $2d-3$ vertices. The facet $J$ has more vertices than $F$ and misses a vertex of largest degree, contrary to our assumption on $F$.}


\green{Assume that $s=d$. By \cref{rmk:equality_case4}(ii), we find that $\deg(v)=2d-2$ and that its vertex figure is a simplicial $(d-1)$-prism $R$. The prism $R$ is the convex hull of two $(d-2)$-simplices $\conv\{w_{1},\ldots,w_{d-1}\}$ and $\conv\{w'_{1},\ldots,w'_{d-1}\}$ so that the vertices $w_{i}$ and $w_{i}'$ are adjacent. For $i\in [1\ldots d-1]$, let $u_{i}$ and $u'_{i}$ represent the vertices in $F\cup F'$ corresponding to the vertices $w_{i}, w_{i}'\in P/v$, respectively.  Moreover, the unique vertex $w\in F\cap F'$ has degree $2d-2$ and is not adjacent to $v$.} 

\green{For $i=1,\ldots,d-1 $, consider the other facet $J_{i}$ in $P$ containing $v$ and arising from the $(d-2)$-face of $R$ missing the edge $u_{i}u_{i}'$. As in \cref{rmk:equality_case4}(ii), each facet $J_{i}$ has $2d-2$ vertices and contains the  vertex $w$. If a vertex $u_{j}$ or $u_{j}'$ has degree $2d-3$ in a facet $J_{i}$ ($i\ne j$), then $J_{i}$ would  be missing a vertex of largest degree in $P$, contradicting our assumption. Thus, no vertex $u_{j}$ or $u_{j}'$ can be an apex of a facet $J_{i}$. Each facet $J_{i}$ is not a pyramid: none of the vertices $u_{j}$ or $u_{j}'$ ($i\ne j$) can serve as  the apex, nor can $v$ or $w$ due to their degrees in $J_{i}$. Additionally, each facet $J_{i}$ is not  a simple polytope, as both $v$ and $w$ have degree $2d-4$ in $J_{i}$.  By \cref{lem:dplus2facets}, each facet $J_{i}$ has at least $d+2$ $(d-2)$-faces.} 


\green{Suppose that, for some $i\in [1\ldots d-1]$, the vertex $u_i$ or $u_i'$ is not simple in $P$. Then, $J_{i}$ is a facet with $2d-2$ vertices and at least $d+2$ $(d-2)$-faces, and since $f_{k}(P)=f_{k}(\Pm(d-1,1))$ for some $k\in [1\ldots d-2]$, the computation \eqref{eq:subcase2.1} in Subcase \ref{case:strong-minimiser-2}.\ref{subcase:2.1}   applies and gives that $f_{k}(J_{i})=f_{k}(\Pm(d-2,1))$. In this scenario, the induction hypothesis on $d-1$ ensures that $J_{i}$ is one of the nonpyramidal polytopes in Parts (ii)--(v) of the theorem, namely $Z(d-1)$, $\TA(d-1)$, or $\WP$ (\cref{fig:polytopes}). In $Z(d-1)$ and $\TA(d-1)$, all the nonsimple vertices are pairwise adjacent, while in $J_i$, the vertex $v$ is not adjacent to $w$. If $J_i$ were $\WP$, then $d=5$ and the degree of $v$ in $J_i$ would be $2d-4=6$, but no vertex in $\WP$ has degree six. Hence, this scenario does not arise.}

%

    

\green{We can now assume that all the vertices $u_i, u_i'$ are simple in $P$. Thus, there are exactly two nonsimple vertices in $P$, namely $v$ and $w$. The graph of $P$ is determined: there are two disjoint $(d-2)$-simplices $\conv\{u_{1},\ldots,u_{d-1}\}$  and $\conv\{u'_{1},\ldots,u'_{d-1}\}$ whose vertices are all adjacent to both $v$ and $w$. 
 However, this is not the graph of a $d$-polytope with $d\ge 5$, as the removal of $v$ and $w$ would disconnect the graph, contradicting the vertex $d$-connectivity of the graph by Balinski's theorem \cite{Bal61}.}
		
 
\end{proof}


\begin{thebibliography}{51}

\bibitem{Bal61}
M.~L. Balinski, \emph{On the graph structure of convex polyhedra in
  {$n$}-space}, {Pac. J. Math.} \textbf{11} (1961), 431--434. \MR{0126765 (23
  \#A4059)}



\bibitem{Bar71}
D.~W. Barnette, {The minimum number of vertices of a simple polytope},
  \emph{Isr. J. Math.} \textbf{10} (1971), 121--125. \MR{0298553 (45 \#7605)}

\bibitem{Bar73}
D.~W. Barnette, {A proof of the lower bound conjecture for convex
  polytopes}, \emph{Pac. J. Math.} \textbf{46} (1973), 349--354. \MR{0328773 (48
  \#7115)}

\bibitem{BriDun73}
D. Britton and J. D. Dunitz,
A complete catalogue of polyhedra with eight or fewer vertices.
{\it Acta Cryst. Ser. A \bf29} (1973), 362--371.

\bibitem{Bro83}
A.~Br{\o}ndsted, \emph{An introduction to convex polytopes}, Graduate Texts in
  Mathematics, vol.~90, Springer-Verlag, New York, 1983. \MR{683612
  (84d:52009)}

\bibitem{DooNevPin18}
J.~Doolittle, E.~Nevo, G.~Pineda-Villavicencio, J.~Ugon, and D.~Yost, {On the reconstruction of polytopes}, \emph{Discrete Comput. Geom.} \textbf{61} (2018),
  no.~2, 285--302.

\bibitem{Fir20}
M.~Firsching, {The complete enumeration of 4-polytopes and 3-spheres with
  nine vertices}, \emph{Isr. J. Math.} \textbf{240} (2020), 417--441.

\bibitem{Gru70}
B.~Gr{{\"u}}nbaum, Polytopes, graphs, and complexes. \emph{Bull. Amer. Math. Soc.} \textbf{76} (1970), 1131--1201.

\bibitem{Gru03}
B.~Gr{{\"u}}nbaum, \emph{Convex polytopes}, 2nd ed., Graduate Texts in
  Mathematics, vol. 221, Springer-Verlag, New York, 2003, Prepared and with a
  preface by V. Kaibel, V. Klee and G. M. Ziegler.


\bibitem{McMShe70}
P.~McMullen and G. C. Shephard, \emph{Convex polytopes and the upper bound conjecture}, London Mathematical Society Lecture Note Series 3, Cambridge University Press, London-New York, 1970.



\bibitem{Pin24}
G.~Pineda-Villavicencio, \emph{Polytopes and graphs}, Cambridge Studies in
  Advanced Mathematics, vol. 211, Cambridge Univ. Press, Cambridge, 2024.
   
\bibitem{PinTriYos24}
G.~Pineda-Villavicencio, A.~Tritama, and D.~Yost, \emph{A lower bound theorem
  for $d$-polytopes with $2d+2$ vertices}, preprint 2024.
 
 \bibitem{PinUgoYos16a}
G.~Pineda-Villavicencio, J.~Ugon, and D.~Yost, {The excess degree of a
  polytope}, \emph{ SIAM J. Discrete Math.} \textbf{32} (2018), no.~3, 2011--2046.

\bibitem{PinUgoYos15}
G.~Pineda-Villavicencio, J.~Ugon, and D.~Yost, {Lower bound theorems for
  general polytopes}, \emph{Eur. J. Comb.} \textbf{79} (2019), 27--45.

\bibitem{PinUgoYos22}
G.~Pineda-Villavicencio, J.~Ugon, and D.~Yost, {Minimum number of edges of
  $d$-polytopes with $2d+2$ vertices}, \emph{Electron. J. Combin.} (2022), P3.18.

 
\bibitem{PinYos22}
G.~Pineda-Villavicencio and D.~Yost, {A lower bound theorem for
  $d$-polytopes with $2d+1$ vertices}, \emph{SIAM J. Discrete Math.} \textbf{36}
  (2022), 2920--2941.

\bibitem{PrzYos16}
K.~Przes{\l}awski and D.~Yost, {More indecomposable polyhedra}, \emph{Extr.
  Math.} \textbf{31} (2016), 169--188.
 
 \bibitem{Xue21}
L.~Xue, {A proof of {Gr\"unbaum's} lower bound conjecture for polytopes},
  \emph{Isr. J. Math.} {\bf245} (2021), 991–1000.
  
  \bibitem{Xue22}
L.~Xue, {A lower bound theorem for strongly regular {CW} spheres with up to $2d +
  1$} vertices, \emph{Discrete Comput. Geom.} (2023). https://doi.org/10.1007/s00454-023-00553-6

\bibitem{Zie95}
G.~M. Ziegler, \emph{Lectures on polytopes}, Graduate Texts in Mathematics,
  vol. 152, Springer-Verlag, New York, 1995. 

\end{thebibliography}
\end{document}